\documentclass[11pt]{article}
\usepackage[utf8]{inputenc}
\usepackage{geometry} 
\usepackage{graphicx} 
\usepackage{booktabs} 
\usepackage{array}
\usepackage{paralist}
\usepackage{verbatim}
\usepackage{subfig}
\usepackage{amsmath}
\usepackage{amssymb}
\usepackage{amsthm}
\usepackage{algorithmic}
\usepackage{algorithm}
\usepackage{cite}
\usepackage{fullpage}
\usepackage{algorithmic}
\usepackage{algorithm}
\floatname{algorithm}{Method}
\usepackage{hyperref}
\usepackage{color}
\newtheorem{remark}{Remark}[section]
\newtheorem{proposition}{Proposition}[section]

\newtheorem{bound}{Bound}[section]

\newtheorem{theorem}{Theorem}[section]
\theoremstyle{remark}

\newcommand{\appr}{\text{approx}}

\usepackage{fancyhdr}
\usepackage[nottoc,notlof,notlot]{tocbibind} 
\usepackage[titles,subfigure]{tocloft}

\title{New Analysis and Results for the Frank-Wolfe Method}
\author{Robert M. Freund\thanks{MIT Sloan School of Management, 77 Massachusetts Avenue, Cambridge, MA   02139
(\href{mailto:rfreund@mit.edu}{rfreund@mit.edu}).  This author's research is supported by AFOSR Grant No. FA9550-11-1-0141 and the MIT-Chile-Pontificia Universidad Católica de Chile Seed Fund.}
\and Paul Grigas\thanks{MIT Operations Research Center, 77 Massachusetts Avenue, Cambridge, MA   02139
(\href{mailto:pgrigas@mit.edu}{pgrigas@mit.edu}).  This author's research has been partially supported through NSF Graduate Research Fellowship No. 1122374 and the MIT-Chile-Pontificia Universidad Católica de Chile Seed Fund.}}

\date{June 1, 2014}

\begin{document}
\maketitle

\begin{abstract}We present new results for the Frank-Wolfe method (also known as the conditional gradient method).  We derive computational guarantees for arbitrary step-size sequences, which are then applied to various step-size rules, including simple averaging and constant step-sizes.  We also develop step-size rules and computational guarantees that depend naturally on the warm-start quality of the initial (and subsequent) iterates.  Our results include computational guarantees for both duality/bound gaps and the so-called FW gaps.  Lastly, we present complexity bounds in the presence of approximate computation of gradients and/or linear optimization subproblem solutions.\end{abstract}

\section{Introduction}

The use and analysis of first-order methods in convex optimization has gained a considerable amount of attention in recent years.  For many applications -- such as LASSO regression, boosting/classification, matrix completion, and other machine learning problems -- first-order methods are appealing for a number of reasons. First, these problems are often very high-dimensional and thus, without any special structural knowledge, interior-point methods or other polynomial-time methods are unappealing.  Second, optimization models in many settings are dependent on data that can be noisy or otherwise limited, and it is therefore not necessary or even sensible to require very high-accuracy solutions.  Thus the weaker rates of convergence of first-order methods are typically satisfactory for such applications. Finally, first-order methods are appealing in many applications due to the lower computational burden per iteration, and the structural implications thereof.  Indeed, most first-order methods require, at each iteration, the computation of an exact, approximate, or stochastic (sub)gradient and the computation of a solution to a particular ``simple'' subproblem.  These computations typically scale well with the dimension of the problem and are often amenable to parallelization, distributed architectures, efficient management of sparse data-structures, and the like.\medskip

Our interest herein is the Frank-Wolfe method, which is also referred to as the conditional gradient method.  The original Frank-Wolfe method, developed for smooth convex optimization on a polytope, dates back to Frank and Wolfe \cite{frank-wolfe}, and was generalized to the more general smooth convex objective function over a bounded convex feasible region thereafter, see for example Demyanov and Rubinov \cite{DR1970}, Dunn and Harshbarger \cite{Dunn1978}, Dunn \cite{Dunn1979}, \cite{Dunn1980}, also Levitin and Polyak \cite{levitin-polyak} and Polyak \cite{polyak}.  More recently there has been renewed interest in the Frank-Wolfe method due to some of its properties that we will shortly discuss, see for example Clarkson \cite{clarkson}, Hazan \cite{hazan}, Jaggi \cite{jaggi2013revisiting}, Giesen et al. \cite{jaggi}, and most recently Harchaoui et al. \cite{harchaoui}, Lan \cite{lancomplex} and Temlyakov \cite{temlyakov}.  The Frank-Wolfe method is premised on being able to easily solve (at each iteration) linear optimization problems over the feasible region of interest.   This is in contrast to other first-order methods, such as the accelerated methods of Nesterov \cite{nest05smoothing,nesterovBook}, which are premised on being able to easily solve (at each iteration) certain projection problems defined by a strongly convex prox function.  In many applications, solving a linear optimization subproblem is much simpler than solving the relevant projection subproblem.  Moreover, in many applications the solutions to the linear optimization subproblems are often highly structured and exhibit particular sparsity and/or low-rank properties, which the Frank-Wolfe method is able to take advantage of as follows.  The Frank-Wolfe method solves one subproblem at each iteration and produces a sequence of feasible solutions that are each a convex combination of all previous subproblem solutions, for which one can derive an $O(\frac{1}{k})$ rate of convergence for appropriately chosen step-sizes.  Due to the structure of the subproblem solutions and the fact that iterates are convex combinations of subproblem solutions, the feasible solutions returned by the Frank-Wolfe method are also typically very highly-structured.  For example, when the feasible region is the unit simplex $\Delta_n := \{\lambda \in \mathbb{R}^n : e^T\lambda = 1, \lambda \geq 0\}$ and the linear optimization oracle always returns an extreme point, then the Frank-Wolfe method has the following sparsity property: the solution that the method produces at iteration $k$ has at most $k$ non-zero entries.  (This observation generalizes to the matrix optimization setting:  if the feasible region is a ball induced by the nuclear norm, then at iteration $k$ the rank of the matrix produced by the method is at most $k$.)  In many applications, such structural properties are highly desirable, and in such cases the Frank-Wolfe method may be more attractive than the faster accelerated methods, even though the Frank-Wolfe method has a slower rate of convergence.\medskip

The first set of contributions in this paper concern computational guarantees for arbitrary step-size sequences.  In Section \ref{sectfw}, we present a new complexity analysis of the Frank-Wolfe method wherein we derive an exact functional dependence of the complexity bound at iteration $k$ as a function of the step-size sequence $\{\bar{\alpha}_k\}$.  We derive bounds on the deviation from the optimal objective function value (and on the duality gap in the presence of minmax structure), and on the so-called FW gaps, which may be interpreted as specially structured duality gaps.  In Section \ref{sect-stepsize}, we use the technical theorems developed in Section \ref{sectfw} to derive computational guarantees for a variety of simple step-size rules including the well-studied step-size rule $\bar{\alpha}_k := \frac{2}{k+2}$, simple averaging, and constant step-sizes.  Our analysis retains the well-known optimal $O(\frac{1}{k})$ rate (optimal for linear optimization oracle-based methods \cite{lancomplex}, following also from \cite{jaggi2013revisiting}) when the step-size is either given by the rule $\bar{\alpha}_k := \frac{2}{k+2}$ or is determined by a line-search.  We also derive an $O\left(\frac{\ln(k)}{k}\right)$ rate for both the case when the step-size is given by simple averaging and in the case when the step-size is simply a suitably chosen constant.\medskip

The second set of contributions in this paper concern ``warm-start'' step-size rules and associated computational guarantees that reflect the the quality of the given initial iterate.  The $O(\frac{1}{k})$ computational guarantees associated with the step-size sequence $\bar{\alpha}_k := \frac{2}{k+2}$ are independent of quality of the initial iterate.  This is good if the objective function value of the initial iterate is very far from the optimal value, as  the computational guarantee is independent of the poor quality of the initial iterate.  But if the objective function value of the initial iterate is moderately close to the optimal value, one would want the Frank-Wolfe method, with an appropriate step-size sequence, to have computational guarantees that reflect the closeness to optimality of the initial objective function value.  In Section \ref{sect-warmstart}, we introduce a modification of the $\bar{\alpha}_k := \frac{2}{k+2}$ step-size rule that incorporates the quality of the initial iterate.  Our new step-size rule maintains the $O(\frac{1}{k})$ complexity bound but now the bound is enhanced by the quality of the initial iterate.  We also introduce a dynamic version of this warm start step-size rule, which dynamically incorporates all new bound information at each iteration.  For the dynamic step-size rule, we also derive a $O(\frac{1}{k})$ complexity bound that depends naturally on all of the bound information obtained throughout the course of the algorithm.\medskip

The third set of contributions concern computational guarantees in the presence of approximate computation of gradients and linear optimization subproblem solutions.  In Section \ref{sect-approx}, we first consider a variation of the Frank-Wolfe method where the linear optimization subproblem at iteration $k$ is solved approximately to an (additive) absolute accuracy of $\delta_k$.  We show that, independent of the choice of step-size sequence $\{\bar{\alpha}_k\}$, the Frank-Wolfe method does not suffer from an accumulation of errors in the presence of approximate subproblem solutions.  We extend the ``technical'' complexity theorems of Section \ref{sectfw}, which imply, for instance, that when an optimal step-size such as $\bar{\alpha}_k := \frac{2}{k+2}$ is used and the $\{\delta_k\}$ accuracy sequence is a constant $\delta$, then a solution with accuracy $O(\frac{1}{k} + \delta)$ can be achieved in $k$ iterations. We next examine variations of the Frank-Wolfe method where exact gradient computations are replaced with inexact gradient computations, under two different models of inexact gradient computations.  We show that all of the complexity results under the previously examined approximate subproblem solution case (including, for instance, the non-accumulation of errors) directly apply to the case where exact gradient computations are replaced with the $\delta$-oracle approximate gradient model introduced by d'Aspremont \cite{daspremontapprox}.  We also examine replacing exact gradient computations with the $(\delta, L)$-oracle model introduced by Devolder et al. \cite{devolderApprox}.  In this case the Frank-Wolfe method suffers from an accumulation of errors under essentially any step-size sequence $\{\bar{\alpha}_k\}$.  These results provide some insight into the inherent tradeoffs faced in choosing among several first-order methods.\medskip

\subsection{Notation}
Let $E$ be a finite-dimensional real vector space with dual vector space $E^\ast$. For a given $s \in E^\ast$ and a given $\lambda \in E$, let $s^T\lambda$ denote the evaluation of the linear functional $s$ at $\lambda$.  For a norm $\|\cdot\|$ on $E$, let $B(c,r) = \{\lambda \in E : \|\lambda - c\| \leq r\}$.  The dual norm $\|\cdot\|_\ast$ on the space $E^\ast$ is defined by $\|s\|_\ast := \max\limits_{\lambda \in B(0,1)} \{s^T\lambda\} $ for a given $s \in E^\ast$.  The notation ``$\tilde v \leftarrow \arg\max\limits_{v  \in S} \{f(v)\}$'' denotes assigning $\tilde v$ to be any optimal solution of the problem $\max\limits_{v  \in S} \{f(v)\}$.

\section{The Frank-Wolfe Method}\label{sectfw}

We recall the Frank-Wolfe method for convex optimization, see Frank and Wolfe \cite{frank-wolfe}, also Demyanov and Rubinov \cite{DR1970}, Levitin and Polyak \cite{levitin-polyak}, and Polyak \cite{polyak}, stated here for maximization problems:
\begin{equation}\label{poi3}\begin{array}{rl}\max\limits_{\lambda} & h(\lambda) \\
\mathrm{s.t.} & \lambda \in Q \ , \end{array}\end{equation} where $Q \subset E$ is convex and compact, and $h(\cdot) : Q \to \mathbb{R}$ is concave and differentiable on $Q$.  Let $h^*$ denote the optimal objective function value of (\ref{poi3}).  The basic Frank-Wolfe method is presented in Method \ref{fw-basic}, where the main computational requirement at each iteration is to solve a linear optimization problem over $Q$ in step (2.) of the method.  The step-size $\bar \alpha_k$ in step (4.) could be chosen by inexact or exact line-search, or by a pre-determined or dynamically determined step-size sequence $\{\bar \alpha_k\}$.  Also note that the version of the Frank-Wolfe method in Method \ref{fw-basic} does not allow a (full) step-size $\bar \alpha_k = 1$, the reasons for which will become apparent below.\medskip

\begin{algorithm}
\caption{Frank-Wolfe Method for maximizing $h(\lambda)$}\label{fw-basic}
\begin{algorithmic}
\STATE Initialize at $\lambda_1 \in Q$, (optional) initial upper bound $B_0$, $k \gets 1$ .

At iteration $k$:
\STATE 1. Compute $\nabla h (\lambda_k)$ .
\STATE 2. Compute $\tilde \lambda_k \gets \arg\max\limits_{\lambda \in Q}\{h(\lambda_k) + \nabla h (\lambda_k)^T(\lambda - \lambda_k)\}$ .\\
\ \ \ \ \ \ \ \ \ \ $B^w_k \leftarrow h(\lambda_k) + \nabla h (\lambda_k)^T(\tilde \lambda_k - \lambda_k)$ .\\
\ \ \ \ \ \ \ \ \ \ $G_k \leftarrow \nabla h (\lambda_k)^T(\tilde \lambda_k - \lambda_k)$ .\\
\STATE 3. (Optional: compute other upper bound $B^o_k$), update best bound $B_k \leftarrow \min\{B_{k-1}, B^w_k, B^o_k\}$ .
\STATE 4. Set $\lambda_{k+1} \gets \lambda_k + \bar{\alpha}_k(\tilde \lambda_k - \lambda_k)$, where $\bar{\alpha}_k \in [0,1)$ .
\end{algorithmic}
\end{algorithm}

As a consequence of solving the linear optimization problem in step (2.) of the method, one conveniently obtains the following upper bound on the optimal value $h^*$ of (\ref{poi3}): \begin{equation}\label{wolfebound} B^w_k := h(\lambda_k) + \nabla h (\lambda_k)^T(\tilde \lambda_k - \lambda_k) \ , \end{equation}and it follows from the fact that the linearization of $h(\cdot)$ at $\lambda_k$ dominates $h(\cdot)$ that $B^w_k$ is a valid upper bound on $h^*$.  We also study the quantity $G_k$ :
\begin{equation}\label{wolfegap} G_k := B^w_k - h(\lambda_k) = \nabla h (\lambda_k)^T(\tilde \lambda_k - \lambda_k) \ , \end{equation}
which we refer to as the ``FW gap'' at iteration $k$ for convenience.  Note that $G_k \ge h^\ast - h(\lambda_k) \ge 0$.  The use of the upper bound $B^w_k$ dates to the original 1956 paper of Frank and Wolfe \cite{frank-wolfe}.  As early as 1970, Demyanov and Rubinov \cite{DR1970} used the FW gap quantities extensively in their convergence proofs of the Frank-Wolfe method, and perhaps this quantity was used even earlier.  In certain contexts, $G_k$ is an important quantity by itself, see for example Hearn \cite{Hearn}, Khachiyan \cite{k} and Giesen et al. \cite{jaggi}. Indeed, Hearn \cite{Hearn} studies basic properties of the FW gaps independent of their use in any algorithmic schemes.  For results concerning upper bound guarantees on $G_k$ for specific and general problems see Khachiyan \cite{k}, Clarkson \cite{clarkson}, Hazan \cite{hazan}, Jaggi \cite{jaggi2013revisiting}, Giesen et al. \cite{jaggi}, and Harchaoui et al. \cite{harchaoui}. Both $B^w_k$ and $G_k$ are computed directly from the solution of the linear optimization problem in step (2.) and are recorded therein for convenience.\medskip

In some of our analysis of the Frank-Wolfe method, the computational guarantees will depend on the quality of upper bounds on $h^*$.  In addition to the Wolfe  bound $B^w_k$, step (3.) allows for an ``optional other upper bound $B^o_k$ '' that also might be computed at iteration $k$.  Sometimes there is structural knowledge of an upper bound as a consequence of a dual problem associated with (\ref{poi3}), as when $h(\cdot)$ is conveyed with minmax structure, namely:
\begin{equation}\label{minmax_structure}
h(\lambda) = \min_{x \in P} \ \phi(x,\lambda) \ ,
\end{equation}
where $P$ is a closed convex set and $\phi(\cdot,\cdot) : P \times Q \to \mathbb{R}$ is a continuous function that is convex in the first variable $x$ and concave in the second variable $\lambda$.  In this case define the convex function $f(\cdot): P \rightarrow \mathbb{R}$ given by $f(x) := \max\limits_{\lambda \in Q} \ \phi(x,\lambda)$ and consider the following duality paired problems:

\begin{equation}
\label{poi2} \mbox{(Primal):} \ \ \min\limits_{x \in P}f(x) \ \ \ \  \ \mbox{and} \ \ \ \ \ \mbox{(Dual):} \ \  \max\limits_{\lambda \in Q}h(\lambda) \ ,
\end{equation}
where the dual problem corresponds to our problem of interest (\ref{poi3}).  Weak duality holds, namely $h(\lambda) \leq h^* \le  f(x)$ for all $x \in P, \lambda \in Q$.  At any iterate $\lambda_k \in Q$ of the Frank-Wolfe method one can construct a ``minmax'' upper bound on $h^*$ by considering the variable $x$ in that structure:
\begin{equation}\label{minmaxbound} B^m_k := f(x_k):=\max\limits_{\lambda \in Q} \{\phi(x_k,\lambda)\} \ \ \ \mbox{where} \ \ \ x_k \in \arg\min\limits_{x \in P}  \{\phi(x,\lambda_k) \} \ , \end{equation}
and it follows from weak duality that $B^o_k := B^m_k$ is a valid upper bound for all $k$.  Notice that $x_k$ defined above is the ``optimal response'' to $\lambda_k$ in a minmax sense and hence is a natural choice of duality-paired variable associated with the variable $\lambda_k$.  Under certain regularity conditions, for instance when $h(\cdot)$ is globally differentiable on $E$, one can show that $B^m_k$ is at least as tight a bound as Wolfe's bound, namely $B^m_k \le B^w_k$ for all $k$ (see Proposition \ref{twobounds}), and therefore the FW gap $G_k$ conveniently bounds this minmax duality gap:  $B^m_k - h(\lambda_k) \le B^w_k - h(\lambda_k) = G_k$.\medskip

(Indeed, in the minmax setting notice that the optimal response $x_k$ in (\ref{minmaxbound}) is a function of the current iterate $\lambda_k$ and hence $f(x_k) - h(\lambda_k)=B^m_k - h(\lambda_k)$ is not just {\em any} duality gap but rather is determined completely by the current iterate $\lambda_k$.  This special feature of the duality gap $B^m_k - h(\lambda_k)$ is exploited in the application of the Frank-Wolfe method to rounding of polytopes \cite{k}, parametric optimization on the spectrahedron \cite{jaggi}, and to regularized regression \cite{gmf2013} (and perhaps elsewhere as well), where bounds on the FW gap $G_k$ are used to bound $B^m_k - h(\lambda_k)$ directly.)\medskip

We also mention that in some applications there might be exact knowledge of the optimal value $h^*$, such as in certain linear regression and/or machine learning applications where one knows {\em a priori} that the optimal value of the loss function is zero.  In these situations one can set $B^o_k \leftarrow h^*$.\medskip

Towards stating and proving complexity bounds for the Frank-Wolfe method, we use the following curvature constant $C_{h, Q}$, which is defined to be the minimal value of $C$ satisfying:
\begin{equation}\label{clarkson} h(\lambda + \alpha (\bar \lambda - \lambda)) \ge h(\lambda) + \nabla h(\lambda)^T(\alpha (\bar \lambda - \lambda)) - \tfrac{1}{2}C \alpha^2 \ \ \ \ \mbox{for~all~} \lambda, \bar \lambda \in Q \mbox{~~and~all~}\alpha \in [0,1] \ .
\end{equation}  (This notion of curvature was introduced by Clarkson \cite{clarkson} and extended in Jaggi \cite{jaggi2013revisiting}.)  For any choice of norm $\|\cdot\|$ on $E$, let $\mathrm{Diam}_Q$ denote the diameter of $Q$ measured with the norm $\|\cdot\|$, namely $\mathrm{Diam}_Q:=\max\limits_{\lambda, \bar \lambda \in Q} \{\|\lambda - \bar \lambda\|\}$ and let $L_{h,Q}$ be the Lipschitz constant for $\nabla h(\cdot)$ on $Q$, namely $L_{h,Q}$ is the smallest constant $L$ for which it holds that:
$$\| \nabla h(\lambda) -\nabla h(\bar \lambda)\|_* \le L \|\lambda - \bar \lambda\| \ \ \ \ \mbox{for~all~} \lambda, \bar \lambda \in Q \ . $$
It is straightforward to show that $C_{h, Q}$ is bounded above by the more classical metrics $\mathrm{Diam}_Q$ and $L_{h,Q}$, namely \begin{equation}\label{clarkbound}C_{h, Q} \le L_{h,Q}(\mathrm{Diam}_Q)^2 \ , \end{equation} see \cite{jaggi2013revisiting}; we present a short proof of this inequality in Proposition \ref{barL} for completeness.  In contrast to other (proximal) first-order methods, the Frank-Wolfe method does not depend on a choice of norm.  The norm invariant definition of $C_{h,Q}$ and the fact that \eqref{clarkbound} holds for \emph{any} norm are therefore particularly appealing properties of $C_{h,Q}$ as a behavioral measure for the Frank-Wolfe method.\medskip

As a prelude to stating our main technical results, we define the following two auxiliary sequences, where $\alpha_k$ and $\beta_k$ are functions of the first $k$ step-size sequence values, $\bar{\alpha}_1, \ldots, \bar{\alpha}_k$, from the Frank-Wolfe method:
\begin{equation}\label{dos}
\beta_k = \frac{1}{\prod\limits_{j=1}^{k-1} (1-\bar \alpha_j)} \ , \ \ \ \ \ \ \alpha_k = \frac{\beta_k \bar \alpha_k}{1- \bar \alpha_k} \ , \ \ \ \ \ k \ge 1 \ .
\end{equation}
(Here and in what follows we use the conventions: $\prod_{j = 1}^0\cdot = 1$ and $\sum_{i = 1}^0\cdot = 0$ .)\medskip

The following two theorems are our main technical constructs that will be used to develop the results herein.  The first theorem concerns optimality gap bounds.

\begin{theorem}\label{fw-complexity}
Consider the iterate sequences of the Frank-Wolfe method (Method \ref{fw-basic}) $\{\lambda_k\}$ and $\{\tilde \lambda_k\}$ and the sequence of upper bounds $\{B_k\}$ on $h^*$, using the step-size sequence $\{\bar\alpha_k\}$.  For the auxiliary sequences $\{\alpha_k\}$ and $\{\beta_k\}$ given by (\ref{dos}), and for any $k \geq 0$, the following inequality holds:
\begin{equation}\label{fw-ineq1}
B_k - h(\lambda_{k+1}) \leq \frac{B_k - h(\lambda_1)}{\beta_{k+1}} + \frac{\frac{1}{2}C_{h,Q}\sum_{i = 1}^k\frac{\alpha_i^2}{\beta_{i+1}}}{\beta_{k+1}} \ .
\end{equation}\qed
\end{theorem}\medskip

\noindent (The summation expression in the rightmost term above appears also in the bound given for the dual averaging method of Nesterov \cite{nesterovda}.  Indeed, this is no coincidence as the sequences $\{\alpha_k\}$ and $\{\beta_k\}$ given by (\ref{dos}) arise precisely from a connection between the Frank-Wolfe method and the dual averaging method. If we define $s_k := \lambda_0 + \sum_{i = 0}^{k-1}\alpha_i\tilde{\lambda}_i$, then one can interpret the sequence $\{s_k\}$ as the sequence of dual variables in a particular instance of the dual averaging method. This connection underlies the proof of Theorem \ref{fw-complexity}, and the careful reader will notice the similarities between the proof of Theorem \ref{fw-complexity} and the proof of Theorem 1 in \cite{nesterovda}. For this reason we will henceforth refer to the sequences (\ref{dos}) as the ``dual averages'' sequences associated with $\{\bar \alpha_k\}$.)\medskip

The second theorem concerns the FW gap values $G_k$ from step (2.) in particular.

\begin{theorem}\label{gap-complexity}
Consider the iterate sequences of the Frank-Wolfe method (Method \ref{fw-basic}) $\{\lambda_k\}$ and $\{\tilde \lambda_k\}$, the sequence of upper bounds $\{B_k\}$ on $h^*$, and the sequence of FW gaps $\{G_k\}$ from step (2.), using the step-size sequence $\{\bar\alpha_k\}$.  For the auxiliary sequences $\{\alpha_k\}$ and $\{\beta_k\}$ given by (\ref{dos}), and for any $\ell \geq 0$ and $k \geq \ell + 1$, the following inequality holds:
\begin{equation}\label{gap-ineq1}
\min_{i \in \{\ell+1, \ldots, k\}}G_i \leq \frac{1}{\sum_{i = \ell+1}^k\bar{\alpha}_i}\left[\frac{B_\ell - h(\lambda_1)}{\beta_{\ell+1}} + \frac{\frac{1}{2}C_{h,Q}\sum_{i = 1}^\ell\frac{\alpha_i^2}{\beta_{i+1}}}{\beta_{\ell+1}}\right] + \frac{\frac{1}{2}C_{h,Q}\sum_{i = \ell+1}^k\bar{\alpha}_i^2}{\sum_{i=\ell+1}^k\bar{\alpha}_i} \ .
\end{equation}\qed
\end{theorem}\medskip

Theorems \ref{fw-complexity} and \ref{gap-complexity} can be applied to yield specific complexity results for {\em any} specific step-size sequence $\{\bar{\alpha}_k\}$ (satisfying the mild assumption that $\bar{\alpha}_k < 1$) through the use of the implied $\{\alpha_k\}$ and $\{\beta_k\}$ dual averages sequences.  This is shown for several useful step-size sequences in the next section.\medskip

\noindent {\bf Proof of Theorem \ref{fw-complexity}:}  We will show the slightly more general result for $k \ge 0$:
\begin{equation}\label{fw-uno}
\min\{B,B_k\} - h(\lambda_{k+1}) \leq \frac{B - h(\lambda_1)}{\beta_{k+1}} + \frac{\frac{1}{2}C_{h,Q}\sum_{i = 1}^k\frac{\alpha_i^2}{\beta_{i+1}}}{\beta_{k+1}} \  \ \ \ \ \ \mathrm{for~any~} B \ ,
\end{equation}from which \eqref{fw-ineq1} follows by substituting $B=B_k$ above.\medskip

For $k=0$ the result follows trivially since $\beta_1=1$ and the summation term on the right side of \eqref{fw-uno} is zero by the conventions for null products and summations stated earlier.  For $k \ge 1$, we begin by observing that the following equalities hold for the dual averages sequences \eqref{dos}:

\begin{equation}\label{identity1} \beta_{i+1} - \beta_i  = \bar \alpha_i \beta_{i+1} = \alpha_i \ \ \ \mathrm{and} \ \ \
\beta_{i+1} \bar\alpha_i^2  = \frac{\alpha_i^2}{\beta_{i+1}} \ \ \ \mathrm{for~} i \ge 1  \ , \end{equation}and
\begin{equation}\label{identity2}1+\sum_{i=1}^k \alpha_i  = \beta_{k+1} \ \ \ \mathrm{for~} k \ge 1  \ . \end{equation}

\noindent We then have for $i \ge 1$:
\begin{align*}
\beta_{i+1} h(\lambda_{i+1}) & \ge \beta_{i+1} \left[h(\lambda_i) + \nabla h(\lambda_i)^T(\tilde\lambda_i-\lambda_i)\bar\alpha_i - \frac{1}{2}\bar\alpha_i^2C_{h,Q} \right] \\
& = \beta_{i} h(\lambda_i) + (\beta_{i+1} - \beta_{i})h(\lambda_i) +  \beta_{i+1}\bar\alpha_i\nabla h(\lambda_i)^T(\tilde\lambda_i-\lambda_i) - \frac{1}{2}\beta_{i+1}\bar\alpha_i^2C_{h,Q} \\
& = \beta_{i} h(\lambda_i) + \alpha_{i}h(\lambda_i) +  \alpha_i\nabla h(\lambda_i)^T(\tilde\lambda_i-\lambda_i) - \frac{1}{2}\frac{\alpha_i^2}{\beta_{i+1}}C_{h,Q} \\
& = \beta_{i} h(\lambda_i) + \alpha_{i}\left[h(\lambda_i) +  \nabla h(\lambda_i)^T(\tilde\lambda_i-\lambda_i)\right] - \frac{1}{2}\frac{\alpha_i^2}{\beta_{i+1}}C_{h,Q} \\
& = \beta_{i} h(\lambda_i) + \alpha_{i}B^w_i - \frac{1}{2}\frac{\alpha_i^2}{\beta_{i+1}}C_{h,Q} \ . \\
\end{align*}  The inequality in the first line above follows from the definition of $C_{h,Q}$ in \eqref{clarkson} and $\lambda_{i+1} -\lambda_i = \bar\alpha_i(\tilde\lambda_i - \lambda_i)$.  The second equality above uses the identities \eqref{identity1}, and the fourth equality uses the definition of the Wolfe upper bound \eqref{wolfebound}.  Rearranging and summing the above over $i$, it follows that for any scalar $B$:
\begin{align}\label{oops}
B+\sum_{i=1}^k \alpha_i B^w_i \le B + \beta_{k+1} h(\lambda_{k+1}) - \beta_1 h(\lambda_1) + \frac{1}{2}\sum_{i=1}^k \frac{\alpha_i^2}{\beta_{i+1}}C_{h,Q} \ .
\end{align}Therefore
\begin{align*}
\min\{B,B_k\}\beta_{k+1} & = \min\{B,B_k\}\left(1+\sum_{i=1}^k \alpha_i \right) \\
& \le  B + \sum_{i=1}^k \alpha_i B^w_i \\
& \le B + \beta_{k+1} h(\lambda_{k+1}) - h(\lambda_1) + \frac{1}{2}\sum_{i=1}^k \frac{\alpha_i^2}{\beta_{i+1}}C_{h,Q} \ ,
\end{align*} where the first equality above uses identity \eqref{identity2}, the first inequality uses the fact that $B_k \le B^w_i$ for $i \le k$, and the second inequality uses \eqref{oops} and the fact that $\beta_1=1$.  The result then follows by dividing by $\beta_{k+1}$ and rearranging terms. \qed\medskip

\noindent {\bf Proof of Theorem \ref{gap-complexity}:}  For $i\ge 1$ we have:
\begin{equation}\begin{array}{rl}\label{duo} h(\lambda_{i+1}) & \ge h(\lambda_i) + \nabla h(\lambda_i)^T(\tilde\lambda_i-\lambda_i)\bar\alpha_i - \frac{1}{2}\bar\alpha_i^2C_{h,Q} \\ \\
& = h(\lambda_i) + \bar\alpha_{i}G_i - \frac{1}{2}\bar\alpha_i^2C_{h,Q} \ ,
\end{array}\end{equation}where the inequality follows from the definition of the curvature constant in \eqref{clarkson}, and the equality follows from the definition of the FW gap in \eqref{wolfegap}.  Summing the above over $i \in \{\ell + 1, \ldots, k\}$ and rearranging yields:
\begin{equation}\begin{array}{rl}\label{trio}
\sum_{i=\ell+1}^k \bar\alpha_{i}G_i & \le h(\lambda_{k+1}) - h(\lambda_{\ell+1}) +  \sum_{i=\ell+1}^k\frac{1}{2}\bar\alpha_i^2C_{h,Q}  \ .
\end{array}
\end{equation}
Combining \eqref{trio} with Theorem \ref{fw-complexity} we obtain:
$$\sum_{i=\ell+1}^k \bar\alpha_{i}G_i  \le h(\lambda_{k+1}) - B_\ell + \frac{B_\ell - h(\lambda_1)}{\beta_{\ell+1}} + \frac{\frac{1}{2}C_{h,Q}\sum_{i = 1}^\ell\frac{\alpha_i^2}{\beta_{i+1}}}{\beta_{\ell+1}} +  \sum_{i=\ell+1}^k\frac{1}{2}\bar\alpha_i^2C_{h,Q} \ , $$
and since $B_\ell \geq h^\ast \geq h(\lambda_{k+1})$ we obtain:
$$\left(\min_{i \in \{\ell+1, \ldots, k\}}G_i\right)\left(\sum_{i=\ell+1}^k \bar\alpha_{i}\right) \le \sum_{i=\ell+1}^k \bar\alpha_{i}G_i  \le \frac{B_\ell - h(\lambda_1)}{\beta_{\ell+1}} + \frac{\frac{1}{2}C_{h,Q}\sum_{i = 1}^\ell\frac{\alpha_i^2}{\beta_{i+1}}}{\beta_{\ell+1}} +  \sum_{i=\ell+1}^k\frac{1}{2}\bar\alpha_i^2C_{h,Q} \ , $$and dividing by $\sum_{i=\ell+1}^k \bar\alpha_{i}$ yields the result. \qed\medskip

\section{Computational Guarantees for Specific Step-size Sequences}\label{sect-stepsize}

Herein we use Theorems \ref{fw-complexity} and \ref{gap-complexity} to derive computational guarantees for a variety of specific step-size sequences.\medskip

It will be useful to consider a version of the Frank-Wolfe method wherein there is a single ``pre-start'' step.  In this case we are given some $\lambda_0 \in Q$ and some upper bound $B_{-1}$ on $h^*$ (one can use $B_{-1}=+\infty$ if no information is available) and we proceed like any other iteration except that in step (4.) we set $\lambda_{1} \gets \tilde \lambda_0$, which is equivalent to setting $\bar \alpha_0 := 1$. This is shown formally in the Pre-start Procedure \ref{prestart}.\medskip

\floatname{algorithm}{Procedure}
\begin{algorithm}
\caption{Pre-start Step of Frank-Wolfe Method given $\lambda_0 \in Q$ and (optional) upper bound $B_{-1}$}\label{prestart}
\begin{algorithmic}
\STATE 1. Compute $\nabla h (\lambda_0)$ .
\STATE 2. Compute $\tilde \lambda_0 \gets \arg\max\limits_{\lambda \in Q}\{h(\lambda_0) + \nabla h (\lambda_0)^T(\lambda - \lambda_0)\}$ .\\
\ \ \ \ \ \ \ \ \ \ $B^w_0 \leftarrow h(\lambda_0) + \nabla h (\lambda_0)^T(\tilde \lambda_0 - \lambda_0)$ .\\
\ \ \ \ \ \ \ \ \ \ $G_0 \leftarrow \nabla h (\lambda_0)^T(\tilde \lambda_0 - \lambda_0)$ .\\
\STATE 3. (Optional: compute other upper bound $B^o_0$), update best bound $B_0 \leftarrow \min\{B_{-1}, B^w_0, B^o_0\}$ .
\STATE 4. Set $\lambda_{1} \gets \tilde \lambda_0$ .
\end{algorithmic}
\end{algorithm}
\floatname{algorithm}{Method}

Before developing computational guarantees for specific step-sizes, we present a property of the pre-start step (Procedure \ref{prestart}) that has implications for such computational guarantees.\medskip

\begin{proposition}\label{prestart-property}  Let $\lambda_1$ and $B_0$ be computed by the pre-start step Procedure \ref{prestart}.  Then $B_0 - h(\lambda_1) \le \frac{1}{2}C_{h,Q}$.
\end{proposition}

\begin{proof}We have $\lambda_1 = \tilde\lambda_0$ and $B_0 \le B^w_0$, whereby from the definition of $C_{h,Q}$ using $\alpha = 1$ we have:
$$h(\lambda_1) = h(\tilde\lambda_0) \ge h(\lambda_0) + \nabla h(\lambda_0)^T(\tilde\lambda_0 - \lambda_0) - \tfrac{1}{2}C_{h,Q} = B^w_0 - \tfrac{1}{2}C_{h,Q} \ge B_0- \tfrac{1}{2}C_{h,Q} \ , $$ and the result follows by rearranging terms.
\end{proof}

\subsection{A Well-studied Step-size Sequence}\label{vivek}
Suppose we initiate the Frank-Wolfe method with the pre-start step Procedure \ref{prestart} from a given value $\lambda_0 \in Q$ (which by definition assigns the step-size $\bar \alpha_0=1$ as discussed earlier), and then use the step-size $\bar \alpha_i = 2/(i+2)$ for $i \ge 1$.  This can be written equivalently as:
\begin{equation}\label{nemirovsteps} \bar \alpha_i = \frac{2}{i+2} \ \ \ \ \ \ \ \mbox{for~} i \ge 0 \ . \end{equation}Computational guarantees for this sequence appeared in Clarkson \cite{clarkson}, Hazan \cite{hazan} (with a corrected proof in Giesen et al. \cite{jaggi}), and Jaggi \cite{jaggi2013revisiting}.  In unpublished correspondence with the first author in 2007, Nemirovski \cite{nemprivate} presented a short inductive proof of convergence of the Frank-Wolfe method using this step-size rule.\medskip

We use the phrase ``bound gap'' to generically refer to the difference between an upper bound $B$ on $h^*$ and the value $h(\lambda)$, namely $B-h(\lambda)$.  The following result describes guarantees on the bound gap $B_k - h(\lambda_{k+1})$ and the FW gap $G_k$ using the step-size sequence \eqref{nemirovsteps}, that are applications of Theorems \ref{fw-complexity} and \ref{gap-complexity}, and that are very minor improvements of existing results as discussed below.

\begin{bound}\label{nemirov-complexity}
Under the step-size sequence (\ref{nemirovsteps}), the following inequalities hold for all $k \ge 1$:
\begin{equation}\label{nemirov-ineq1}
B_k - h(\lambda_{k+1}) \leq \frac{2C_{h,Q}}{k+4}
\end{equation}
and
\begin{equation}\label{nemirov-gap-ineq1}
\ \  \ \ \ \ \ \ \min_{i \in \{1, \ldots, k\}}G_i \leq \frac{4.5C_{h,Q}}{k} \ .
\end{equation}
\end{bound}\medskip

\noindent The bound \eqref{nemirov-ineq1} is a very minor improvement over that in Hazan \cite{hazan}, Giesen et al. \cite{jaggi}, Jaggi \cite{jaggi2013revisiting}, and Harchaoui et al. \cite{harchaoui}, as the denominator is additively larger by $1$ (after accounting for the pre-start step and the different indexing conventions).  The bound \eqref{nemirov-gap-ineq1} is a modification of the original bound in Jaggi \cite{jaggi2013revisiting}, and is also a slight improvement of the bound in Harchaoui et al. \cite{harchaoui} inasmuch as the denominator is additively larger by $1$ and the bound is valid for all $k \ge 1$.\medskip

\noindent {\bf Proof of Bound \ref{nemirov-complexity}:}  Using \eqref{nemirovsteps} it is easy to show that the dual averages sequences \eqref{dos} satisfy $\beta_k = \frac{k(k+1)}{2}$ and $\alpha_k = k+1$ for $k\ge 1$.  Utilizing Theorem \ref{fw-complexity}, we have for $k\ge 1$:
\begin{align*}
B_k - h(\lambda_{k+1}) & \leq \frac{B_k - h(\lambda_1)}{\beta_{k+1}} + \frac{\frac{1}{2}C_{h,Q}\sum_{i = 1}^k\frac{\alpha_i^2}{\beta_{i+1}}}{\beta_{k+1}} \\
& \leq \frac{B_0 - h(\lambda_1)}{\beta_{k+1}} + \frac{\frac{1}{2}C_{h,Q}\sum_{i = 1}^k\frac{\alpha_i^2}{\beta_{i+1}}}{\beta_{k+1}} \\
& \leq \frac{\frac{1}{2}C_{h,Q}}{\beta_{k+1}} + \frac{\frac{1}{2}C_{h,Q}\sum_{i = 1}^k\frac{\alpha_i^2}{\beta_{i+1}}}{\beta_{k+1}} \\
& = \frac{C_{h,Q}}{(k+1)(k+2)}\left[1+ \sum_{i = 1}^k\frac{2(i+1)^2}{(i+1)(i+2)}\right] \\
& = \frac{C_{h,Q}}{(k+1)(k+2)}\left[\sum_{i = 0}^k\frac{2(i+1)}{(i+2)}\right] \\
& \le \frac{2C_{h,Q}}{k+4} \ ,
\end{align*}where the second inequality uses $B_k \le B_0$, the third inequality uses Proposition \ref{prestart-property}, the first equality substitutes the dual averages sequence values, and the final inequality follows from Proposition \ref{boringbound1}.  This proves \eqref{nemirov-ineq1}.\medskip

To prove \eqref{nemirov-gap-ineq1} we proceed as follows.  First apply Theorem \ref{gap-complexity} with $\ell = 0$ and $k = 1$ to obtain:
\begin{equation*}
G_1 \leq \frac{1}{\bar{\alpha}_1}\left[B_0 - h(\lambda_1)\right] + \frac{1}{2}C_{h,Q}\bar{\alpha}_1 \leq \frac{1}{2}C_{h,Q}\left[\frac{1}{\bar{\alpha}_1} + \bar{\alpha}_1\right] = \frac{1}{2}C_{h,Q}\left[\frac{3}{2} + \frac{2}{3}\right] = \frac{13}{12}C_{h,Q} \ ,
\end{equation*}
where the second inequality uses Proposition \ref{prestart-property}. Since $\frac{13}{12} \leq 4.5$ and $\frac{13}{12} \leq \frac{4.5}{2}$, this proves \eqref{nemirov-gap-ineq1} for $k = 1, 2$. Assume now that $k \geq 3$. Let $\ell = \lceil\frac{k}{2}\rceil - 2$ so that $\ell \geq 0$. We have:
\begin{equation}\label{nem_lb}
\sum_{i = \ell+1}^k\bar{\alpha}_i = 2\sum_{i = \ell+1}^k\frac{1}{i + 2} = 2\sum_{i = \ell + 3}^{k+2}\frac{1}{i} \geq 2\ln\left(\frac{k+3}{\ell+3}\right) \geq 2\ln\left(\frac{k+3}{\frac{k}{2} + 1.5}\right) = 2\ln(2) \ ,
\end{equation}
where the first inequality uses Proposition \ref{boringbound3} and the second inequality uses $\lceil\frac{k}{2}\rceil \leq \frac{k}{2} + \frac{1}{2}$. We also have:
\begin{equation}\label{nem_ub}
\sum_{i = \ell+1}^k\bar{\alpha}_i^2 = 4\sum_{i = \ell+1}^k\frac{1}{(i+2)^2} = 4\sum_{i = \ell+3}^{k+2}\frac{1}{i^2} \leq \frac{4(k - \ell)}{(\ell + 2)(k+2)} \leq \frac{4\left(\frac{k}{2} + 2\right)}{\frac{k}{2}(k+2)} = \frac{4(k + 4)}{k(k+2)} \ ,
\end{equation}
where the first inequality uses Proposition \ref{boringbound3} and the second inequality uses $\lceil\frac{k}{2}\rceil \geq \frac{k}{2}$. Applying Theorem \ref{gap-complexity} and using \eqref{nem_lb} and \eqref{nem_ub} yields:
\begin{align*}
\min_{i \in \{1, \ldots, k\}}G_i &\leq \frac{1}{2\ln(2)}\left[\frac{B_\ell - h(\lambda_1)}{\beta_{\ell+1}} + \frac{\frac{1}{2}C_{h,Q}\sum_{i = 1}^\ell\frac{\alpha_i^2}{\beta_{i+1}}}{\beta_{\ell+1}} + \frac{2C_{h,Q}(k + 4)}{k(k+2)} \right]\\
&\leq \frac{1}{2\ln(2)}\left[\frac{2C_{h,Q}}{\ell + 4} + \frac{2C_{h,Q}(k + 4)}{k(k+2)} \right]\\
&\leq \frac{2C_{h,Q}}{2\ln(2)}\left[\frac{2}{k + 4} + \frac{k + 4}{k(k+2)} \right] = \frac{2C_{h,Q}}{2\ln(2)}\left[\frac{3k^2 + 12k + 16}{(k+4)(k+2)k} \right] \leq \frac{2C_{h,Q}}{2\ln(2)}\left(\frac{3}{k}\right) \leq \frac{4.5C_{h,Q}}{k} \ ,
\end{align*}
where the second inequality uses the chain of inequalities used to prove \eqref{nemirov-ineq1}, the third inequality uses $\ell + 4 \geq \frac{k}{2} + 2$, and the fourth inequality uses $k^2 + 4k + \frac{16}{3} \leq k^2 + 6k + 8 = (k+4)(k+2)$.\qed

\subsection{Simple Averaging}\label{pointer}
Consider the following step-size sequence:
\begin{equation}\label{avgsteps}
\bar \alpha_i = \frac{1}{i+1} \ \ \ \ \ \ \ \  \ \mbox{for~} i \ge 0 \ ,
\end{equation}
where, as with the step-size sequence \eqref{nemirovsteps}, we write $\bar{\alpha}_0 = 1$ to indicate the use of the pre-start step Procedure \ref{prestart}. It follows from a simple inductive argument that, under the step-size sequence \eqref{avgsteps}, $\lambda_{k+1}$ is the simple average of $\tilde{\lambda}_0, \tilde{\lambda}_1, \ldots, \tilde{\lambda}_k$, i.e., we have
\begin{equation*}
\lambda_{k+1} = \frac{1}{k+1}\sum_{i = 0}^k\tilde{\lambda}_i \ \ \ \  \ \mbox{for all~} k \geq 0 \ .
\end{equation*}

\begin{bound}\label{average-complexity}
Under the step-size sequence \eqref{avgsteps}, the following inequality holds for all $k \ge 0$:
\begin{equation}\label{average-ineq}
B_{k} - h(\lambda_{k+1}) \le \frac{\frac{1}{2}C_{h,Q}(1+\ln(k+1))}{k+1} \ ,
\end{equation}
and the following inequality holds for all $k \geq 2$:
\begin{equation}\label{average-gap-ineq}
\min_{i \in \{1, \ldots, k\}}G_i \leq \frac{\frac{3}{4}C_{h,Q}\left(2.3 + 2\ln(k)\right)}{k-1} \ .
\end{equation}
\end{bound}\medskip

\noindent {\bf Proof of Bound \ref{average-complexity}:}
Using \eqref{avgsteps} it is easy to show that the dual averages sequences \eqref{dos} are given by $\beta_k = k$ and $\alpha_k = 1$ for $k\ge 1$.  Utilizing Theorem \ref{fw-complexity} and Proposition \ref{prestart-property}, we have for $k\ge 1$:
\begin{align*}
B_k - h(\lambda_{k+1}) & \leq \frac{\frac{1}{2}C_{h,Q}}{\beta_{k+1}} + \frac{\frac{1}{2}C_{h,Q}\sum_{i = 1}^k\frac{\alpha_i^2}{\beta_{i+1}}}{\beta_{k+1}} \\
& = \frac{\frac{1}{2}C_{h,Q}}{k+1}\left[1+ \sum_{i = 1}^k\frac{1}{i+1}\right] \\
& \le \frac{\frac{1}{2}C_{h,Q}}{k+1}\left[1 + \ln(k+1)\right] \ ,
\end{align*}
where the first equality substitutes the dual averages sequence values and the second inequality uses Proposition \ref{boringbound3}. This proves \eqref{average-ineq}. To prove \eqref{average-gap-ineq}, we proceed as follows. Let $\ell = \lfloor\frac{k}{2}\rfloor - 1$, whereby $\ell \geq 0$ since $k \geq 2$. We have:
\begin{equation}\label{avg_lb}
\sum_{i = \ell+1}^k\bar{\alpha}_i = \sum_{i = \ell+1}^k\frac{1}{i + 1} = \sum_{i = \ell + 2}^{k+1}\frac{1}{i} \geq \ln\left(\frac{k+2}{\ell + 2}\right) \geq \ln\left(\frac{k+2}{\frac{k}{2} + 1}\right) = \ln(2) \ ,
\end{equation}
where the first inequality uses Proposition \ref{boringbound3} and the second inequality uses $\ell \leq \frac{k}{2} - 1$. We also have:
\begin{equation}\label{avg_ub}
\sum_{i = \ell+1}^k\bar{\alpha}_i^2 = \sum_{i = \ell+1}^k\frac{1}{(i+1)^2} = \sum_{i = \ell+2}^{k+1}\frac{1}{i^2} \leq \frac{k - \ell}{(\ell + 1)(k + 1)} \leq \frac{\frac{k}{2} + 1.5}{(\frac{k}{2} - \frac{1}{2})(k+1)} = \frac{k+3}{(k-1)(k+1)} \ ,
\end{equation}
where the first inequality uses Proposition \ref{boringbound3} and the second inequality uses $\ell \geq \frac{k}{2} - 1.5$. Applying Theorem \ref{gap-complexity} and using \eqref{avg_lb} and \eqref{avg_ub} yields:
\begin{align*}
\min_{i \in \{1, \ldots, k\}}G_i &\leq \frac{1}{\ln(2)}\left[\frac{B_\ell - h(\lambda_1)}{\beta_{\ell+1}} + \frac{\frac{1}{2}C_{h,Q}\sum_{i = 1}^\ell\frac{\alpha_i^2}{\beta_{i+1}}}{\beta_{\ell+1}} + \frac{\frac{1}{2}C_{h,Q}(k+3)}{(k-1)(k+1)}\right]\\
&\leq \frac{1}{\ln(2)}\left[\frac{\frac{1}{2}C_{h,Q}(1+\ln(\ell+1))}{\ell+1}  + \frac{\frac{1}{2}C_{h,Q}(k+3)}{(k-1)(k+1)}\right]\\
&\leq \frac{\frac{1}{2}C_{h,Q}}{\ln(2)}\left[\frac{1+\ln(\frac{k}{2})}{\frac{k}{2} - \frac{1}{2}}  + \frac{k+3}{(k-1)(k+1)}\right]\\
&\leq \frac{\frac{1}{2}C_{h,Q}}{\ln(2)}\left[\frac{2 + 2\ln(k) - 2\ln(2)}{k - 1}  + \frac{\frac{5}{3}}{k-1}\right]\\
&\leq \frac{\frac{3}{4}C_{h,Q}\left(2.3 + 2\ln(k)\right)}{k-1} \ ,
\end{align*}
where the second inequality uses the bound that proves \eqref{average-ineq}, the third inequality uses $\frac{k}{2} - 1.5 \leq \ell \leq \frac{k}{2} - 1$ and the fourth inequality uses $\frac{k+3}{k+1} \leq \frac{5}{3}$ for $k \geq 2$.\qed

\subsection{Constant Step-size}\label{brick}
Given $\bar \alpha \in (0,1)$, consider using the following constant step-size rule:
\begin{equation}\label{constantstep}
\bar \alpha_i = \bar{\alpha} \ \ \ \ \  \ \mbox{for~} i \ge 1 \ .
\end{equation}This step-size rule arises in the analysis of the Incremental Forward Stagewise Regression algorithm ($\mathrm{FS}_\varepsilon$), see  \cite{gmf2013}, and perhaps elsewhere as well.
\begin{bound}\label{constant-bound}
Under the step-size sequence (\ref{constantstep}), the following inequality holds for all $k \ge 1$:
\begin{equation}\label{constant-bound1}
 B_{k} - h(\lambda_{k+1})  \leq \left(B_k - h(\lambda_1)\right)(1-\bar{\alpha})^{k}+ \tfrac{1}{2}C_{h,Q}\left[ \bar{\alpha}-\bar \alpha(1-\bar{\alpha})^{k} \right] \ .
\end{equation}
If the pre-start step Procedure \ref{prestart} is used, then:
\begin{equation}\label{constant-bound2}
 B_{k} - h(\lambda_{k+1})  \leq \tfrac{1}{2}C_{h,Q}\left[(1-\bar{\alpha})^{k+1} + \bar{\alpha}\right] \ .
\end{equation}
\end{bound}
\noindent If we decide {\em a priori} to run the Frank-Wolfe method for $k$ iterations after the pre-start step Procedure \ref{prestart}, then we can optimize the bound \eqref{constant-bound2} with respect to $\bar{\alpha}$.  The optimized value of $\bar{\alpha}$ in the bound \eqref{constant-bound2} is easily derived to be:
\begin{equation}\label{opt-alpha}
\bar{\alpha}^{\ast} = 1 - \frac{1}{\sqrt[k]{k+1}} \ .
\end{equation}
With $\bar{\alpha}$ determined by (\ref{opt-alpha}), we obtain a simplified bound from \eqref{constant-bound2} and also a guarantee for the FW gap sequence $\{G_k\}$ if the method is continued with the same constant step-size (\ref{opt-alpha}) for an additional $k + 1$ iterations.
\begin{bound}\label{constant-opt-bound}
If we use the pre-start step Procedure \ref{prestart} and the constant step-size sequence \eqref{opt-alpha} for all iterations, then after $k$ iterations the following inequality holds:
\begin{equation}\label{constant-opt-bound1}
B_{k} - h(\lambda_{k+1}) \le \frac{\frac{1}{2}C_{h,Q}\left(1 + \ln(k+1)\right)}{k} \ .
\end{equation}
Furthermore, after $2k + 1$ iterations the following inequality holds:
\begin{equation}\label{constant-opt-bound2}
\min_{i \in \{1, \ldots, 2k+1\}} G_i \leq \frac{\frac{1}{2}C_{h,Q}\left(1 + 2\ln(k+1)\right)}{k}
\end{equation}
\end{bound}
\noindent It is curious to note that the bounds \eqref{average-ineq} and \eqref{constant-opt-bound1} are almost identical, although \eqref{constant-opt-bound1} requires fixing {\em a priori} the number of iterations $k$.\medskip

\noindent {\bf Proof of Bound \ref{constant-bound}:}  Under the step-size rule \eqref{constantstep} it is straightforward to show that the dual averages sequences \eqref{dos} are for $i \ge 1$:
$$ \beta_i = (1-\bar\alpha)^{-k+1}  \ \ \ \ \mathrm{and} \ \ \ \ \alpha_i = \bar\alpha(1-\bar\alpha)^{-k} \ , $$
whereby
$$ \sum_{i=1}^k \frac{\alpha_i^2}{\beta_{i+1}} =  \sum_{i=1}^k \bar\alpha^2(1-\bar\alpha)^{-i} =\bar\alpha^2\left(\frac{\frac{1}{(1-\bar\alpha)^k}-1}{\bar\alpha} \right) = \bar\alpha\left[(1-\bar\alpha)^{-k} -1 \right] \ . $$
It therefore follows from Theorem \ref{fw-complexity} that:
\begin{equation}\label{addon}\begin{array}{rl}
B_k - h(\lambda_{k+1}) & \leq \frac{B_k - h(\lambda_1)}{\beta_{k+1}} + \frac{\frac{1}{2}C_{h,Q}\sum_{i = 1}^k\frac{\alpha_i^2}{\beta_{i+1}}}{\beta_{k+1}} \\ \\
& = \left(B_k - h(\lambda_1)\right)(1-\bar\alpha)^{k} + \left(\frac{C_{h,Q}}{2}\right)\bar\alpha\left[(1-\bar\alpha)^{-k} -1 \right](1-\bar\alpha)^{k} \\ \\
& = \left(B_k - h(\lambda_1)\right)(1-\bar\alpha)^{k} + \left(\frac{C_{h,Q}}{2}\right)\left[\bar\alpha - \bar\alpha (1-\bar\alpha)^{k} \right] \ ,
\end{array}\end{equation} which proves \eqref{constant-bound1}.  If the pre-start step Procedure \ref{prestart} is used, then using Proposition \ref{prestart-property} it follows that $B_k - h(\lambda_1) \le B_0 - h(\lambda_1) \le \frac{1}{2}C_{h,Q}$, whereby from \eqref{constant-bound1} we obtain:
\begin{align*}
B_k - h(\lambda_{k+1}) & \leq  \frac{1}{2}C_{h,Q}(1-\bar\alpha)^{k} + \left(\frac{C_{h,Q}}{2}\right)\left[\bar\alpha - \bar\alpha (1-\bar\alpha)^{k} \right] \\ \\
& =   \frac{1}{2}C_{h,Q} \left[(1-\bar\alpha)^{k+1} + \bar\alpha \right] \ ,
\end{align*}
completing the proof. \qed\medskip

\noindent {\bf Proof of Bound \ref{constant-opt-bound}:}  Substituting the step-size \eqref{opt-alpha} into \eqref{constant-bound2} we obtain:
\begin{align*}
B_k - h(\lambda_{k+1}) & \leq \frac{1}{2}C_{h,Q} \left[\left(\frac{1}{\sqrt[k]{k+1}}\right)^{k+1} +  1 - \frac{1}{\sqrt[k]{k+1}} \right] \\ \\
& \leq \frac{1}{2}C_{h,Q} \left[\left(\frac{1}{\sqrt[k]{k+1}}\right)^{k+1} + \frac{\ln(k+1)}{k} \right] \\ \\
& \leq \frac{1}{2}C_{h,Q} \left[\frac{1}{k+1} + \frac{\ln(k+1)}{k} \right] \\ \\
& \leq \frac{1}{2}C_{h,Q} \left[\frac{1}{k} + \frac{\ln(k+1)}{k} \right] \ ,
\end{align*}where the second inequality follows from {\em (i)} of Proposition \ref{boringbound2}.  This proves \eqref{constant-opt-bound1}.  To prove \eqref{constant-opt-bound2}, notice that inequality \eqref{addon} together with the subsequent chain of inequalities in the proofs of \eqref{constant-bound1}, \eqref{constant-bound2}, and \eqref{constant-opt-bound1} show that:
\begin{equation}\label{late1} \left[\frac{B_k - h(\lambda_1)}{\beta_{k+1}} + \frac{\frac{1}{2}C_{h,Q}\sum_{i = 1}^k\frac{\alpha_i^2}{\beta_{i+1}}}{\beta_{k+1}}\right] \le \frac{1}{2}C_{h,Q} \left( \frac{1+\ln(k+1)}{k} \right) \ . \end{equation}
Using \eqref{late1} and the substitution $\sum_{i=k+1}^{2k+1}\bar\alpha_i = (k+1)\bar\alpha$ and $\sum_{i=k+1}^{2k+1}\bar\alpha_i^2 = (k+1)\bar\alpha^2$ in Theorem \ref{gap-complexity} yields:
\begin{align*}
\min_{i \in \{1, \ldots, 2k+1\}}G_i & \le \frac{1}{(k+1)\bar\alpha}\left( \frac{\tfrac{1}{2}C_{h,Q}(1+\ln(k+1))}{k} \right) + \frac{\frac{1}{2}C_{h,Q}(k+1)\bar\alpha^2}{(k+1)\bar\alpha} \\ \\
& \le \tfrac{1}{2}C_{h,Q}\left( \frac{1+\ln(k+1)}{k} \right) + \tfrac{1}{2}C_{h,Q} \cdot \bar\alpha \\ \\
& \le \frac{\frac{1}{2}C_{h,Q}\left(1 + 2\ln(k+1)\right)}{k} \ ,
\end{align*}where the second inequality uses {\em (ii)} of Proposition \ref{boringbound2} and the third inequality uses {\em (i)} of Proposition \ref{boringbound2}. \qed

\subsection{Extensions using Line-Searches}
The original method of Frank and Wolfe \cite{frank-wolfe} utilized an exact line-search to determine the next iterate $\lambda_{k+1}$ by assigning $\hat\alpha_k \gets \arg\max\limits_{\alpha \in [0,1]}\{h(\lambda_k + \alpha (\tilde\lambda_k - \lambda_k)) \}$ and $\lambda_{k+1} \gets \lambda_k + \hat\alpha_k(\tilde\lambda_k - \lambda_k)$.  When $h(\cdot)$ is a quadratic function and the dimension of the space $E$ of variables $\lambda$ is not huge, an exact line-search is easy to compute analytically.  It is a straightforward extension of Theorem \ref{fw-complexity} to show that if an exact line-search is utilized at every iteration, then the bound \eqref{fw-ineq1} holds for \emph{any} choice of step-size sequence $\{\bar\alpha_k\}$, and not just the sequence $\{\hat\alpha_k\}$ of line-search step-sizes. In particular, the $O(\frac{1}{k})$ computational guarantee \eqref{nemirov-ineq1} holds, as does \eqref{average-ineq} and \eqref{constant-bound1}, as well as the bound \eqref{warmstartbound2} to be developed in Section \ref{sect-warmstart}.  This observation generalizes as follows. At iteration $k$ of the Frank-Wolfe method, let $A_k \subseteq [0,1)$ be a closed set of potential step-sizes and suppose we select the next iterate $\lambda_{k+1}$ using the exact line-search assignment $\hat\alpha_k \gets \arg\max\limits_{\alpha \in A_k}\{h(\lambda_k + \alpha (\tilde\lambda_k - \lambda_k)) \}$ and $\lambda_{k+1} \gets \lambda_k + \hat\alpha_k(\tilde\lambda_k - \lambda_k)$. Then after $k$ iterations of the Frank-Wolfe method, we can apply the bound \eqref{fw-ineq1} for any choice of step-size sequence $\{\bar\alpha_i\}_{i = 1}^k$ in the cross-product $A_1 \times \cdots \times A_k$.\medskip

For inexact line-search methods, Dunn \cite{Dunn1980} analyzes versions of the Frank-Wolfe method with an Armijo line-search and also a Goldstein line-search rule.  In addition to convergence and computational guarantees for convex problems, \cite{Dunn1980} also contains results for the case when the objective function is non-concave.  And in prior work, Dunn \cite{Dunn1979} presents convergence and computational guarantees for the case when the step-size $\bar\alpha_k$ is determined from the structure of the lower quadratic approximation of $h(\cdot)$ in \eqref{clarkson}, if the curvature constant $C_{h,Q}$ is known or upper-bounded.  And in the case when no prior information about $C_{h,Q}$ is given, \cite{Dunn1979} has a clever recursion for determining a step-size that still accounts for the lower quadratic approximation without estimation of $C_{h,Q}$.
\medskip

\section{Computational Guarantees for a Warm Start}\label{sect-warmstart}

In the framework of this study, the well-studied step-size sequence \eqref{nemirovsteps} and associated computational guarantees (Bound \ref{nemirov-complexity}) corresponds to running the Frank-Wolfe method initiated with the pre-start step from the initial point $\lambda_0$.  One feature of the main computational guarantees as presented in the bounds \eqref{nemirov-ineq1} and \eqref{nemirov-gap-ineq1} is their insensitivity to the quality of the initial point $\lambda_0$.  This is good if $h(\lambda_0)$ is very far from the optimal value $h^\ast$, as the poor quality of the initial point does not affect the computational guarantee.  But if $h(\lambda_0)$ is moderately close to the optimal value, one would want the Frank-Wolfe method, with an appropriate step-size sequence, to have computational guarantees that reflect the closeness to optimality of the initial objective function value $h(\lambda_0)$.  Let us see how this can be done.\medskip

We will consider starting the Frank-Wolfe method {\em without} the pre-start step, started at an initial point $\lambda_1$, and let $C_1$ be a given {\em estimate} of the curvature constant $C_{h,Q}$.  Consider the following step-size sequence:
\begin{equation}\label{warmstartstep}
\bar \alpha_i = \displaystyle\frac{2}{\frac{2C_1}{B_1-h(\lambda_1)}+i+1} \ \ \ \ \  \ \mbox{for~} i \ge 1 \ .
\end{equation}
Comparing \eqref{warmstartstep} to the well-studied step-size rule \eqref{nemirovsteps}, one can think of the above step-size rule as acting ``as if'' the Frank-Wolfe method had run for $\frac{2C_1}{B_1-h(\lambda_1)} $ iterations before arriving at $\lambda_1$.  The next result presents a computational guarantee associated with this step-size rule.
\begin{bound}\label{warmstartbound}
Under the step-size sequence (\ref{warmstartstep}), the following inequality holds for all $k \ge 1$:
\begin{equation}\label{warmstartbound1}
B_{k} - h(\lambda_{k+1})  \leq \displaystyle\frac{2\max\{C_1, C_{h,Q}\}}{\frac{2C_1}{B_1-h(\lambda_1)} \ + \  k} \ .
\end{equation}
\end{bound}
\noindent Notice that in the case when $C_1=C_{h,Q}$, the bound in \eqref{warmstartbound1} simplifies conveniently to:
\begin{equation}\label{warmstartbound2}
B_{k} - h(\lambda_{k+1})  \leq \displaystyle\frac{2C_{h,Q}}{\frac{2C_{h,Q}}{B_1-h(\lambda_1)} \ + \  k} \ .
\end{equation}
\noindent Also, as a function of the estimate $C_1$ of the curvature constant, it is easily verified that the bound in \eqref{warmstartbound1} is optimized at $C_1=C_{h,Q}$.\medskip

We remark that the bound \eqref{warmstartbound1} (or \eqref{warmstartbound2}) is small to the extent that the initial bound gap $B_1-h(\lambda_1)$ is small, as one would want.  However, to the extent that $B_1-h(\lambda_1)$ is small, the incremental decrease in the bound due to an additional iteration is less.  In other words, while the bound \eqref{warmstartbound1} is nicely sensitive to the initial bound gap, there is no longer rapid decrease in the bound in the early iterations.  It is as if the algorithm had already run for $\left(\frac{2C_1}{B_1 - h(\lambda_1)}\right)$ iterations to arrive at the initial iterate $\lambda_1$, with a corresponding dampening in the marginal value of each iteration after then.  This is a structural feature of the Frank-Wolfe method that is different from first-order methods that use prox functions and/or projections.\medskip

\noindent {\bf Proof of Bound \ref{warmstartbound}:}  Define $s = \frac{2C_1}{B_1 - h(\lambda_1)}$, whereby $\bar\alpha_i = \frac{2}{s+1+i}$ for $i \ge 1$.  It then is straightforward to show that the dual averages sequences \eqref{dos} are for $i \ge 1$:
$$\beta_i = \prod\limits_{j=1}^{i-1} (1-\bar\alpha_j)^{-1} = \prod\limits_{j=1}^{i-1} \frac{s+j+1}{s+j-1} = \frac{(s+i-1)(s+i)}{s(s+1)} \ , $$ and
$$\alpha_i = \frac{\beta_i \bar\alpha_i}{1-\bar\alpha_i} = \frac{2(s+i)(s+i-1)(s+i+1)}{s(s+1)(s+i+1)(s+i-1)} = \frac{2(s+i)}{s(s+1)} \ . $$
Furthermore, we have:
\begin{equation}\label{cerro} \sum_{i=1}^k\frac{\alpha_i^2}{\beta_{i+1}} = \sum_{i=1}^k\frac{4(s+i)^2(s)(s+1)}{s^2(s+1)^2(s+i)(s+i+1)} = \sum_{i=1}^k\frac{4(s+i)}{s(s+1)(s+i+1)} \le \frac{4k}{s(s+1)} \ .
\end{equation}Utilizing Theorem \ref{fw-complexity} and \eqref{cerro}, we have for $k\ge 1$:
\begin{align*}
B_k - h(\lambda_{k+1}) & \leq \frac{B_k - h(\lambda_1)}{\beta_{k+1}} + \frac{\frac{1}{2}C_{h,Q}\sum_{i = 1}^k\frac{\alpha_i^2}{\beta_{i+1}}}{\beta_{k+1}} \\ \\
& \leq \frac{s(s+1)}{(s+k)(s+k+1)}\left(B_1 - h(\lambda_1)+ \frac{C_{h,Q}}{2}\cdot\frac{4k}{s(s+1)}\right) \\ \\
& = \frac{s(s+1)}{(s+k)(s+k+1)}\left(\frac{2C_1}{s}+ \frac{2 k C_{h,Q}}{s(s+1)}\right) \\ \\
& \le \frac{2\max\{C_1,C_{h,Q}\}}{(s+k)(s+k+1)}\left(s+1+k \right)\\ \\
& = \frac{2\max\{C_1,C_{h,Q}\}}{s+k} \ ,
\end{align*} which completes the proof. \qed

\subsection{A Dynamic Version of the Warm-Start Step-size Strategy}\label{voip}

The step-size sequence \eqref{warmstartstep} determines all step-sizes for the Frank-Wolfe method based on two pieces of information at the initial point $\lambda_1$: {\em (i)} the initial bound gap $B_1 - h(\lambda_1)$, and {\em (ii)} the given estimate $C_1$ of the curvature constant.  The step-size sequence \eqref{warmstartstep} is a static warm-start strategy in that all step-sizes are determined by information that is available or computed at the first iterate.  Let us see how we can improve the computational guarantee by treating every iterate as if it were the initial iterate, and hence dynamically determine the steps-size sequence as a function of accumulated information about the bound gap and the curvature constant.\medskip

At the start of a given iteration $k$ of the Frank-Wolfe method, we have the iterate value $\lambda_k \in Q$ and an upper bound $B_{k-1}$ on $h^*$ from the previous iteration.  We also will now assume that we have an estimate $C_{k-1}$ of the curvature constant from the previous iteration as well.  Steps (2.) and (3.) of the Frank-Wolfe method then perform the computation of $\tilde \lambda_k$, $B_k$ and $G_k$.  Instead of using a pre-set formula for the step-size $\bar\alpha_k$, we will determine the value of $\bar\alpha_k$ based on the current bound gap $B_k-h(\lambda_k)$ as well as on a new estimate $C_{k}$ of the curvature constant. (We will shortly discuss how $C_k$ is computed.)  Assuming $C_k$ has been computed, and mimicking the structure of the static warm-start step-size rule \eqref{warmstartstep}, we compute the current step-size as follows:
\begin{equation}\label{Aa}  \bar\alpha_k:=\frac{2}{\frac{2C_k}{B_k-h(\lambda_k)}+2} \ ,
\end{equation}
\noindent where we note that $\bar\alpha_k$ depends explicitly on the value of $C_k$.  Comparing $\bar\alpha_k$ in \eqref{Aa} with \eqref{nemirovsteps}, we interpret $\frac{2C_k}{B_k-h(\lambda_k)}$ to be ``as if'' the current iteration $k$ was preceded by $\frac{2C_k}{B_k-h(\lambda_k)}$ iterations of the Frank-Wolfe method using the standard step-size \eqref{nemirovsteps}.  This interpretation is also in concert with that of the static warm-start step-size rule \eqref{warmstartstep}.\medskip

We now discuss how we propose to compute the new estimate $C_k$ of the curvature constant $C_{h,Q}$ at iteration $k$.  Because $C_k$ will be only an estimate of $C_{h,Q}$, we will need to require that $C_k$ (and the step-size $\bar\alpha_k$ \eqref{Aa} that depends explicitly on $C_k$) satisfy:
\begin{equation}\label{cme}
h(\lambda_k+\bar\alpha_k(\tilde \lambda_k-\lambda_k)) \ge h(\lambda_k) + \bar \alpha_k(B_k-h(\lambda_k))-\frac{1}{2}C_k\bar\alpha_k^2 \ .
\end{equation}
In order to find a value $C_k \ge C_{k-1}$ for which \eqref{cme} is satisfied, we first test if $C_k := C_{k-1}$ satisfies \eqref{cme}, and if so we set $C_k \leftarrow C_{k-1}$.  If not, one can perform a standard doubling strategy, testing values $C_k \leftarrow 2C_{k-1}, 4C_{k-1}, 8C_{k-1}, \ldots$, until \eqref{cme} is satisfied.  Since \eqref{cme} will be satisfied whenever $C_k \ge C_{h,Q}$ from the definition of $C_{h,Q}$ in \eqref{clarkson} and the inequality $B_k-h(\lambda_k) \le B^w_k-h(\lambda_k)=\nabla h(\lambda_k)^T(\tilde\lambda_k - \lambda_k)$, it follows that the doubling strategy will guarantee $C_k \le \max\{C_0, 2C_{h,Q}\}$. Of course, if an upper bound $\bar{C} \geq C_{h,Q}$ is known, then $C_k \gets \bar{C}$ is a valid assignment for all $k \geq 1$. Moreover, the structure of $h(\cdot)$ may be sufficiently simple so that a value of $C_k \ge C_{k-1}$ satisfying \eqref{cme} can be determined analytically via closed-form calculation, as is the case if $h(\cdot)$ is a quadratic function for example.  The formal description of the Frank-Wolfe method with dynamic step-size strategy is presented in Method \ref{fw-warm-iterates}.\medskip

\begin{algorithm}
\caption{Frank-Wolfe Method with Dynamic Step-sizes for maximizing $h(\lambda)$}\label{fw-warm-iterates}
\begin{algorithmic}
\STATE Initialize at $\lambda_1 \in Q$, initial estimate $C_0$ of $C_{h,Q}$, (optional) initial upper bound $B_0$, $k \gets 1$ .

At iteration $k$:
\STATE 1. Compute $\nabla h (\lambda_k)$ .
\STATE 2. Compute $\tilde \lambda_k \gets \arg\max\limits_{\lambda \in Q}\{h(\lambda_k) + \nabla h (\lambda_k)^T(\lambda - \lambda_k)\}$ .\\
\ \ \ \ \ \ \ \ \ \ $B^w_k \leftarrow h(\lambda_k) + \nabla h (\lambda_k)^T(\tilde \lambda_k - \lambda_k)$ .\\
\ \ \ \ \ \ \ \ \ \ $G_k \leftarrow \nabla h (\lambda_k)^T(\tilde \lambda_k - \lambda_k)$ .\\
\STATE 3. (Optional: compute other upper bound $B^o_k$), update best bound $B_k \leftarrow \min\{B_{k-1}, B^w_k, B^o_k\}$ .\\
\STATE 4. Compute $C_k$ for which the following conditions hold:\\
\ \ \ \ \ \ \ \ \ \ {\em (i)} $C_k \ge C_{k-1}$ , and  \\
\ \ \ \ \ \ \ \ \ \ {\em (ii)} $h(\lambda_k+\bar\alpha_k(\tilde \lambda_k-\lambda_k)) \ge h(\lambda_k) + \bar \alpha_k(B_k-h(\lambda_k))-\frac{1}{2}C_k\bar\alpha_k^2$ , where $\bar\alpha_k:=\frac{2}{\frac{2C_k}{B_k-h(\lambda_k)}+2}$ .\\
\STATE 5. Set $\lambda_{k+1} \gets \lambda_k + \bar{\alpha}_k(\tilde \lambda_k - \lambda_k)$ .
\end{algorithmic}
\end{algorithm}

\noindent We have the following computational guarantees for the Frank-Wolfe method with dynamic step-sizes (Method \ref{fw-warm-iterates}):

\begin{bound}\label{warmiteratesbound}
The iterates of the Frank-Wolfe method with dynamic step-sizes (Method \ref{fw-warm-iterates}) satisfy the following for any $k\ge 1$:
\begin{equation}\label{dynamic-conclusion}
B_{k} - h(\lambda_{k})  \leq \min_{\ell \in \{1, \ldots, k\}}\left\{\displaystyle\frac{2C_{k}}{\frac{2C_{k}}{B_\ell-h(\lambda_\ell)} \ + \ k - \ell }\right\} \ .
\end{equation}  Furthermore, if the doubling strategy is used to update the estimates $\{C_k\}$ of $C_{h,Q}$, it holds that $C_{k} \le \max\{C_0, 2C_{h,Q}\}$.
\end{bound}

\noindent Notice that \eqref{dynamic-conclusion} naturally generalizes the static warm-start bound \eqref{warmstartbound1} (or \eqref{warmstartbound2}) to this more general dynamic case. Consider, for simplicity, the case where $C_k = C_{h,Q}$ is the known curvature constant. In this case, \eqref{dynamic-conclusion} says that we may apply the bound \eqref{warmstartbound2} with any $\ell \in \{1, \ldots, k\}$ as the starting iteration. That is, the computational guarantee for the dynamic case is at least as good as the computational guarantee for the static warm-start step-size \eqref{warmstartstep} initialized at \emph{any} iteration $\ell \in \{1, \ldots, k\}$.\medskip

\noindent {\bf Proof of Bound \ref{warmiteratesbound}:}  Let $i \ge 1$.  For convenience define $A_i = \frac{2C_i}{B_i - h(\lambda_i)}$, and in this notation \eqref{Aa} is $\bar\alpha_i = \frac{2}{A_i + 2}$. Applying \emph{(ii)} in step (4.) of Method \ref{fw-warm-iterates} we have:
\begin{align*}
B_{i+1} - h(\lambda_{i+1}) & \le B_{i+1} - h(\lambda_i) - \bar\alpha_i (B_i - h(\lambda_i)) +\tfrac{1}{2}\bar\alpha_i^2C_i \\ \\
& \le B_i - h(\lambda_i) - \bar\alpha_i (B_i - h(\lambda_i)) +\tfrac{1}{2}\bar\alpha_i^2C_i \\ \\
& = (B_i - h(\lambda_i))(1 - \bar\alpha_i ) +\tfrac{1}{2}\bar\alpha_i^2C_i \\ \\
& = \frac{2C_i}{A_i}\left(\frac{A_i}{A_i + 2}\right) + \frac{2C_i}{(A_i + 2)^2} \\ \\
& = 2C_i\left(\frac{A_i + 3}{(A_i + 2)^2}\right) \\ \\
& < \frac{2C_i}{A_i + 1} \ ,
\end{align*} where the last inequality follows from the fact that $(a+2)^2 > a^2 + 4a + 3 = (a+1)(a+3)$ for $a \ge 0$.  Therefore
\begin{equation}\label{Bb}
A_{i+1} = \frac{2C_{i+1}}{B_{i+1} - h(\lambda_{i+1})} = \frac{C_{i+1}}{C_{i}}\left(\frac{2C_{i}}{B_{i+1} - h(\lambda_{i+1})}\right) > \frac{C_{i+1}}{C_{i}}\left( A_i + 1 \right) \ .
\end{equation}We now show by reverse induction that for any $\ell \in \{1, \ldots, k\}$ the following inequality is true:
\begin{equation}\label{induction}
A_{k} \geq \frac{C_{k}}{C_{\ell}}A_\ell + k - \ell \ .
\end{equation}
Clearly \eqref{induction} holds for $\ell=k$, so let us suppose \eqref{induction} holds for some $\ell + 1 \in \{2, \ldots, k\}$. Then
\begin{align*}
A_{k} &\geq \frac{C_{k}}{C_{\ell + 1}}A_{\ell + 1} + k - \ell - 1 \\ \\
& > \frac{C_{k}}{C_{\ell + 1}}\left(\frac{C_{\ell+1}}{C_{\ell}}\left(A_\ell + 1 \right)\right) + k - \ell - 1 \\ \\
& \ge \frac{C_{k}}{C_{\ell}}A_\ell + k - \ell \ ,
\end{align*}
where the first inequality is the induction hypothesis, the second inequality uses \eqref{Bb}, and the third inequality uses the monotonicity of the $\{C_k\}$ sequence.  This proves \eqref{induction}.  Now for any $\ell \in \{1, \ldots, k\}$ we have from \eqref{induction} that:
$$ B_{k} - h(\lambda_{k}) = \frac{2C_{k}}{A_{k}} \leq \frac{2C_{k}}{\frac{C_{k}}{C_{\ell}}A_\ell + k - \ell}  = \frac{2C_{k}}{\frac{2C_{k}}{B_{\ell} - h(\lambda_{\ell})} + k - \ell } \ , $$proving the result. \qed

\section{Analysis of the Frank-Wolfe Method with Inexact Gradient Computations and/or Subproblem Solutions}\label{sect-approx}

In this section we present and analyze extensions of the Frank-Wolfe method in the presence of inexact computation of gradients and/or subproblem solutions.  We first consider the case when the linear optimization subproblem is solved approximately.

\subsection{Frank-Wolfe Method with Inexact Linear Optimization Subproblem Solutions}\label{approx-subprob}

Here we consider the case when the linear optimization subproblem is solved approximately, which arises especially in optimization over matrix variables.  For example, consider instances of \eqref{poi3} where $Q$ is the spectrahedron of symmetric matrices, namely $Q=\{ \Lambda \in \mathbb{S}^{n \times n} : \Lambda \succeq 0, \ I \bullet \Lambda = 1\}$, where $\mathbb{S}^{n \times n}$ is the space of symmetric matrices of order $n$, ``$\succeq$'' is the L\"{o}wner ordering thereon, and ``$\cdot \bullet \cdot$'' denotes the trace inner product.  For these instances solving the linear optimization subproblem corresponds to computing the leading eigenvector of a symmetric matrix, whose solution when $n \gg 0$ is typically computed inexactly using iterative methods.   For $\delta \geq 0$ an (absolute) $\delta$-approximate solution to the linear optimization subproblem $\max\limits_{\lambda \in Q}\left\{c^T\lambda\right\}$ is a vector $\tilde{\lambda} \in Q$ satisfying:
\begin{equation}\label{approx-def}
c^T\tilde{\lambda} \geq \max\limits_{\lambda \in Q}\left\{c^T\lambda\right\} - \delta \ ,
\end{equation}
and we use the notation $\tilde{\lambda} \gets \appr(\delta)_{\lambda \in Q}\left\{c^T\lambda\right\}$ to denote assigning to $\tilde{\lambda}$ any such $\delta$-approximate solution.  The same additive linear optimization subproblem approximation model is  considered in Dunn and Harshbarger \cite{Dunn1978} and Jaggi \cite{jaggi2013revisiting}, and a multiplicative linear optimization subproblem approximation model is considered in Lacoste-Julien et al. \cite{lacoste2012block}; a related approximation model is used in connection with a greedy coordinate descent method in Dud\'{i}k et al. \cite{dudik}.  In Method \ref{fw-approx} we present a version of the Frank-Wolfe algorithm that uses approximate linear optimization subproblem solutions.  Note that Method \ref{fw-approx} allows for the approximation quality $\delta = \delta_k$ to be a function of the iteration index $k$.  Note also that the definition of the Wolfe upper bound $B^w_k$ and the FW gap $G_k$ in step (2.) are amended from the original Frank-Wolfe algorithm (Method \ref{fw-basic}) by an additional term $\delta_k$.  It follows from \eqref{approx-def} that:
\begin{equation*}
B^w_k = h(\lambda_k) + \nabla h (\lambda_k)^T(\tilde \lambda_k - \lambda_k) + \delta_k \geq \max_{\lambda \in Q}\left\{h(\lambda_k) + \nabla h (\lambda_k)^T(\lambda - \lambda_k)\right\} \geq h^\ast \ ,
\end{equation*}which shows that $B^w_k$ is a valid upper bound on $h^*$, with similar properties for $G_k$.  The following two theorems extend Theorem \ref{fw-complexity} and Theorem \ref{gap-complexity} to the case of approximate subproblem solutions.  Analogous to the the case of exact subproblem solutions, these two theorems can easily be used to derive suitable bounds for specific step-sizes rules such as those in Sections \ref{sect-stepsize} and \ref{sect-warmstart}.\medskip

\begin{algorithm}
\caption{Frank-Wolfe Method with Approximate Subproblem Solutions}\label{fw-approx}
\begin{algorithmic}
\STATE Initialize at $\lambda_1 \in Q$, (optional) initial upper bound $B_0$, $k \gets 1$ .

At iteration $k$:
\STATE 1. Compute $\nabla h (\lambda_k)$ .
\STATE 2. Compute $\tilde \lambda_k \gets \appr(\delta_k)_{\lambda \in Q}\{h(\lambda_k) + \nabla h (\lambda_k)^T(\lambda - \lambda_k)\}$ .\\
\ \ \ \ \ \ \ \ \ \ $B^w_k \leftarrow h(\lambda_k) + \nabla h (\lambda_k)^T(\tilde \lambda_k - \lambda_k) + \delta_k$ .\\
\ \ \ \ \ \ \ \ \ \ $G_k \leftarrow \nabla h (\lambda_k)^T(\tilde \lambda_k - \lambda_k) + \delta_k$ .\\
\STATE 3. (Optional: compute other upper bound $B^o_k$), update best bound $B_k \leftarrow \min\{B_{k-1}, B^w_k, B^o_k\}$ .
\STATE 4. Set $\lambda_{k+1} \gets \lambda_k + \bar{\alpha}_k(\tilde \lambda_k - \lambda_k)$, where $\bar{\alpha}_k \in [0,1)$ .
\end{algorithmic}
\end{algorithm}\medskip

\begin{theorem}\label{fwapprox-complexity}
Consider the iterate sequences of the Frank-Wolfe method with approximate subproblem solutions (Method \ref{fw-approx}) $\{\lambda_k\}$ and $\{\tilde \lambda_k\}$ and the sequence of upper bounds $\{B_k\}$ on $h^*$, using the step-size sequence $\{\bar\alpha_k\}$.  For the auxiliary sequences $\{\alpha_k\}$ and $\{\beta_k\}$ given by (\ref{dos}), and for any $k \geq 0$, the following inequality holds:
\begin{equation}\label{fwapprox-ineq1}
B_k - h(\lambda_{k+1}) \leq \frac{B_k - h(\lambda_1)}{\beta_{k+1}} + \frac{\frac{1}{2}C_{h,Q}\sum_{i = 1}^k\frac{\alpha_i^2}{\beta_{i+1}}}{\beta_{k+1}} + \frac{\sum_{i = 1}^k\alpha_i\delta_i}{\beta_{k+1}} \ .
\end{equation}\qed
\end{theorem}

\begin{theorem}\label{gapapprox-complexity}
Consider the iterate sequences of the Frank-Wolfe method with approximate subproblem solutions (Method \ref{fw-approx}) $\{\lambda_k\}$ and $\{\tilde \lambda_k\}$, the sequence of upper bounds $\{B_k\}$ on $h^*$, and the sequence of FW gaps $\{G_k\}$ from step (2.), using the step-size sequence $\{\bar\alpha_k\}$.  For the auxiliary sequences $\{\alpha_k\}$ and $\{\beta_k\}$ given by (\ref{dos}), and for any $\ell \geq 0$ and $k \geq \ell + 1$, the following inequality holds:
\begin{align}\label{gapapprox-ineq1}
\min_{i \in \{\ell+1, \ldots, k\}}G_i \leq \ \ &\frac{1}{\sum_{i = \ell+1}^k\bar{\alpha}_i}\left[\frac{B_\ell - h(\lambda_1)}{\beta_{\ell+1}} + \frac{\frac{1}{2}C_{h,Q}\sum_{i = 1}^\ell\frac{\alpha_i^2}{\beta_{i+1}}}{\beta_{\ell+1}} + \frac{\sum_{i = 1}^\ell\alpha_i\delta_i}{\beta_{\ell+1}}\right]\\
&+ \frac{\frac{1}{2}C_{h,Q}\sum_{i = \ell+1}^k\bar{\alpha}_i^2}{\sum_{i=\ell+1}^k\bar{\alpha}_i} + \frac{\sum_{i = \ell+1}^k\bar{\alpha}_i\delta_i}{\sum_{i=\ell+1}^k\bar{\alpha}_i} \ . \nonumber
\end{align}\qed
\end{theorem}

\begin{remark}\label{prestart-property-approx}  The pre-start step (Procedure \ref{prestart} can also be generalized to the case of approximate solution of the linear optimization subproblem.  Let $\lambda_1$ and $B_0$ be computed by the pre-start step with a $\delta =\delta_0$-approximate subproblem solution.  Then Proposition \ref{prestart-property} generalizes to: $$B_0 - h(\lambda_1) \le \frac{1}{2}C_{h,Q} + \delta_0 \ , $$ and hence if the pre-start step is used \eqref{fwapprox-ineq1} implies that:
\begin{equation}\label{electric}
B_k - h(\lambda_{k+1}) \leq \frac{\frac{1}{2}C_{h,Q}\sum_{i = 0}^k\frac{\alpha_i^2}{\beta_{i+1}}}{\beta_{k+1}} + \frac{\sum_{i = 0}^k\alpha_i\delta_i}{\beta_{k+1}} \ ,
\end{equation}
where $\alpha_0 := 1$.
\end{remark}\medskip

Let us now discuss implications of Theorems \ref{fwapprox-complexity} and \ref{gapapprox-complexity}, and Remark \ref{prestart-property-approx}.  Observe that the bounds on the right-hand sides of \eqref{fwapprox-ineq1} and \eqref{gapapprox-ineq1} are composed of the exact terms which appear on the right-hand sides of \eqref{fw-ineq1} and \eqref{gap-ineq1}, plus additional terms involving the solution accuracy sequence $\delta_1, \ldots, \delta_k$.  It follows from \eqref{identity2} that these latter terms are particular convex combinations of the $\delta_i$ values and zero, and in \eqref{electric} the last term is a convex combination of the $\delta_i$ values, whereby they are trivially bounded above by $\max\{\delta_1, \ldots, \delta_k\}$.  When $\delta_i := \delta$ is a constant, then this bound is simply $\delta$, and we see that the errors due to the approximate computation of linear optimization subproblem solutions do not accumulate, independent of the choice of step-size sequence $\{\bar{\alpha}_k\}$.   In other words, Theorem \ref{fwapprox-complexity} implies that if we are able to solve the linear optimization subproblems to an accuracy of $\delta$, then the Frank-Wolfe method can solve \eqref{poi3} to an accuracy of $\delta$ plus a function of the step-size sequence $\{\bar{\alpha}_k\}$, the latter of which can be made to go to zero at an appropriate rate depending on the choice of step-sizes.  Similar observations hold for the terms depending on $\delta_1, \ldots, \delta_k$ that appear on the right-hand side of \eqref{gapapprox-ineq1}.\medskip

Note that Jaggi \cite{jaggi2013revisiting} considers the case where $\delta_i := \frac{1}{2}\delta\bar{\alpha}_iC_{h,Q}$ (for some fixed $\delta \geq 0$) and $\bar{\alpha}_i := \frac{2}{i+2}$ for $i \geq 0$ (or $\bar{\alpha}_i$ is determined by a line-search), and shows that in this case Method \ref{fw-approx} achieves $O\left(\frac{1}{k}\right)$ convergence in terms of both the optimality gap and the FW gaps.  These results can be recovered as a particular instantiation of Theorems \ref{fwapprox-complexity} and \ref{gapapprox-complexity} using similar logic as in the proof of Bound \ref{nemirov-complexity}.\medskip

\noindent {\bf Proof of Theorem \ref{fwapprox-complexity}:}
First recall the identities \eqref{identity1} and \eqref{identity2} for the dual averages sequences \eqref{dos}.  Following the proof of Theorem \ref{fw-complexity}, we then have for $i \ge 1$:
\begin{align*}
\beta_{i+1} h(\lambda_{i+1}) & \ge \beta_{i+1} \left[h(\lambda_i) + \nabla h(\lambda_i)^T(\tilde\lambda_i-\lambda_i)\bar\alpha_i - \frac{1}{2}\bar\alpha_i^2C_{h,Q} \right] \\
& = \beta_{i} h(\lambda_i) + (\beta_{i+1} - \beta_{i})h(\lambda_i) +  \beta_{i+1}\bar\alpha_i\nabla h(\lambda_i)^T(\tilde\lambda_i-\lambda_i) - \frac{1}{2}\beta_{i+1}\bar\alpha_i^2C_{h,Q} \\
& = \beta_{i} h(\lambda_i) + \alpha_{i}\left[h(\lambda_i) +  \nabla h(\lambda_i)^T(\tilde\lambda_i-\lambda_i)\right] - \frac{1}{2}\frac{\alpha_i^2}{\beta_{i+1}}C_{h,Q} \\
& = \beta_{i} h(\lambda_i) + \alpha_{i}B^w_i - \alpha_i\delta_i - \frac{1}{2}\frac{\alpha_i^2}{\beta_{i+1}}C_{h,Q} \ ,
\end{align*}
where the third equality above uses the definition of the Wolfe upper bound \eqref{wolfebound} in Method \ref{fw-approx}.  The rest of the proof follows exactly as in the proof of Theorem \ref{fw-complexity}. \qed\medskip

\noindent {\bf Proof of Theorem \ref{gapapprox-complexity}:}  For $i\ge 1$ we have:
\begin{equation*}\begin{array}{rl}
h(\lambda_{i+1}) & \ge h(\lambda_i) + \nabla h(\lambda_i)^T(\tilde\lambda_i-\lambda_i)\bar\alpha_i - \frac{1}{2}\bar\alpha_i^2C_{h,Q} \\ \\
& = h(\lambda_i) + \bar\alpha_{i}G_i - \bar\alpha_i\delta_i - \frac{1}{2}\bar\alpha_i^2C_{h,Q} \ ,
\end{array}\end{equation*}
where the equality above follows from the definition of the FW gap in Method \ref{fw-approx}. Summing the above over $i \in \{\ell + 1, \ldots, k\}$ and rearranging yields:
\begin{equation}\begin{array}{rl}\label{trioapp}
\sum_{i=\ell+1}^k \bar\alpha_{i}G_i & \le h(\lambda_{k+1}) - h(\lambda_{\ell+1}) +  \sum_{i=\ell+1}^k\frac{1}{2}\bar\alpha_i^2C_{h,Q} + \sum_{i = \ell+1}^k\bar{\alpha}_i\delta_i \ .
\end{array}
\end{equation}
The rest of the proof follows by combining \eqref{trioapp} with Theorem \ref{fwapprox-complexity} and proceeding as in the proof of Theorem \ref{gap-complexity}. \qed

\subsection{Frank-Wolfe Method with Inexact Gradient Computations}\label{subsect-aspremont}

We now consider a version of the Frank-Wolfe method where the exact gradient computation is replaced with the computation of an approximate gradient, as was explored in Section 3 of Jaggi \cite{jaggi2013revisiting}.  We analyze two different models of approximate gradients and derive computational guarantees for each model.  We first consider the $\delta$-oracle model of d'Aspremont \cite{daspremontapprox}, which was developed in the context of accelerated first-order methods. For $\delta \geq 0$, a $\delta$-oracle is a (possibly non-unique) mapping $g_\delta(\cdot): Q \to E^\ast$ that satisfies:
\begin{equation}\label{dCondition}
\left|(\nabla h(\bar{\lambda}) - g_\delta(\bar{\lambda}))^T(\lambda - \bar{\lambda})\right| \leq \delta \ \ \text{ for all } \lambda, \bar{\lambda} \in Q \ .
\end{equation}
Note that the definition of the $\delta$-oracle does not consider inexact computation of function values.  Depending on the choice of step-size sequence $\{\bar{\alpha}_k\}$, this assumption is acceptable as the Frank-Wolfe method may or may not need to compute function values.  (The warm-start step-size rule \eqref{Aa} requires computing function values, as does the computation of the upper bounds $\{B_k^w\}$, in which case a definition analogous to \eqref{dCondition} for function values can be utilized.)\medskip

The next proposition states the following:  suppose one solves for the exact solution of the linear optimization subproblem using the $\delta$-oracle instead of the exact gradient.  Then the absolute suboptimality of the computed solution in terms of the exact gradient is at most $2\delta$.\medskip

\begin{proposition}\label{dProp}
For any $\bar{\lambda} \in Q$ and any $\delta \geq 0$, if $\tilde{\lambda} \in \arg\max\limits_{\lambda \in Q}\left\{g_\delta(\bar{\lambda})^T\lambda\right\}$, then $\tilde{\lambda}$ is a $2\delta$-approximate solution to the linear optimization subproblem $\max\limits_{\lambda \in Q}\left\{\nabla h(\bar{\lambda})^T\lambda\right\}$.
\end{proposition}
\begin{proof}
Let $\hat{\lambda} \in \arg\max\limits_{\lambda \in Q}\left\{\nabla h(\bar{\lambda})^T\lambda\right\}$. Then, we have:
\begin{align*}
\nabla h(\bar{\lambda})^T(\tilde{\lambda} - \bar{\lambda}) &\geq g_\delta(\bar{\lambda})^T(\tilde{\lambda} - \bar{\lambda}) - \delta\\
&\geq g_\delta(\bar{\lambda})^T(\hat{\lambda} - \bar{\lambda}) - \delta\\
&\geq \nabla h(\bar{\lambda})^T(\hat{\lambda} - \bar{\lambda}) - 2\delta\\
&= \max_{\lambda \in Q}\left\{\nabla h(\bar{\lambda})^T\lambda\right\} - \nabla h(\bar{\lambda})^T\bar{\lambda}) -2\delta \ ,
\end{align*}
where the first and third inequalities use \eqref{dCondition}, the second inequality follows since $\tilde{\lambda} \in \arg\max\limits_{\lambda \in Q}\left\{g_\delta(\bar{\lambda})^T\lambda\right\}$, and the final equality follows since $\hat{\lambda} \in \arg\max\limits_{\lambda \in Q}\left\{\nabla h(\bar{\lambda})^T\lambda\right\}$.  Rearranging terms then yields the result.
\end{proof}

Now consider a version of the Frank-Wolfe method where the computation of $\nabla h(\lambda_k)$ at step (1.) is replaced with the computation of $g_{\delta_k}(\lambda_k)$.  Then Proposition \ref{dProp} implies that such a version can be viewed simply as a special case of the version of the Frank-Wolfe method with approximate subproblem solutions (Method \ref{fw-approx}) of Section \ref{approx-subprob} with $\delta_k$ replaced by $2\delta_k$.  Thus, we may readily apply Theorems \ref{fwapprox-complexity} and \ref{gapapprox-complexity} and Proposition \ref{prestart-property-approx} to this case. In particular, similar to the results in \cite{daspremontapprox} regarding error non-accumulation for an accelerated first-order method, the results herein imply that there is no accumulation of errors for a version of the Frank-Wolfe method that computes approximate gradients with a $\delta$-oracle at each iteration.  Furthermore, it is a simple extension to consider a version of the Frank-Wolfe method that computes both (i) approximate gradients with a $\delta$-oracle, and (ii) approximate linear optimization subproblem solutions.

\subsubsection{Inexact Gradient Computation Model via the $(\delta,L)$-oracle}

The premise \eqref{dCondition} underlying the $\delta$-oracle is quite strong and can be restrictive in many cases.  For this reason among others, Devolder et al. \cite{devolderApprox} introduce the less restrictive $(\delta, L)$-oracle model.  For scalars $\delta, L \geq 0$, the $(\delta, L)$-oracle is defined as a (possibly non-unique) mapping $Q \to \mathbb{R} \times E^\ast$ that maps $\bar{\lambda} \to (h_{(\delta,L)}(\bar{\lambda}), g_{(\delta, L)}(\bar{\lambda}))$ which satisfy:
\begin{align}
h(\lambda) &\le h_{(\delta,L)}(\bar \lambda) + g_{(\delta,L)}(\bar \lambda)^T(\lambda - \bar \lambda) \ , \ \text{ and }\label{dgnCondition}\\
h(\lambda) &\ge h_{(\delta,L)}(\bar \lambda) + g_{(\delta,L)}(\bar \lambda)^T(\lambda - \bar \lambda) - \frac{L}{2}\|\lambda - \bar \lambda\|^2 - \delta \ \ \text{ for all } \lambda, \bar{\lambda} \in Q \label{dgnCondition2}\ ,
\end{align}
where $\|\cdot\|$ is a choice of norm on $E$.  Note that in contrast to the $\delta$-oracle model, the $(\delta, L)$-oracle model does assume that the function $h(\cdot)$ is smooth or even concave -- it simply assumes that there is an oracle returning the pair $(h_{(\delta,L)}(\bar{\lambda}), g_{(\delta, L)}(\bar{\lambda}))$ satisfying \eqref{dgnCondition} and \eqref{dgnCondition2}.\medskip

In Method \ref{fw-dgn} we present a version of the Frank-Wolfe method that utilizes the $(\delta, L)$-oracle.  Note that we allow the parameters $\delta$ and $L$ of the $(\delta, L)$-oracle to be a function of the iteration index $k$.  Inequality \eqref{dgnCondition} in the definition of the $(\delta, L)$-oracle immediately implies that $B^w_k \geq h^\ast$.  We now state the main technical complexity bound for Method \ref{fw-dgn}, in terms of the sequence of bound gaps $\{B_k - h(\lambda_{k+1})\}$.  Recall from Section \ref{sectfw} the definition $\mathrm{Diam}_Q := \max\limits_{\lambda, \bar \lambda \in Q} \{\|\lambda - \bar \lambda\|\}$, where the norm $\|\cdot\|$ is the norm used in the definition of the $(\delta, L)$-oracle \eqref{dgnCondition2}.\medskip

\begin{algorithm}
\caption{Frank-Wolfe Method With $(\delta, L)$-Oracle}\label{fw-dgn}
\begin{algorithmic}
\STATE Initialize at $\lambda_1 \in Q$, (optional) initial upper bound $B_0$, $k \gets 1$ .

At iteration $k$:
\STATE 1. Compute $h_k \gets h_{(\delta_k,L_k)}(\lambda_k), \ g_k \gets g_{(\delta_k, L_k)}(\lambda_k)$ .
\STATE 2. Compute $\tilde \lambda_k \gets \arg\max\limits_{\lambda \in Q}\{h_k + g_k^T(\lambda - \lambda_k)\}$ .\\
\ \ \ \ \ \ \ \ \ \ $B^w_k \leftarrow h_k + g_k^T(\tilde \lambda_k - \lambda_k)$ .\\
\STATE 3. (Optional: compute other upper bound $B^o_k$), update best bound $B_k \leftarrow \min\{B_{k-1}, B^w_k, B^o_k\}$ .
\STATE 4. Set $\lambda_{k+1} \gets \lambda_k + \bar{\alpha}_k(\tilde \lambda_k - \lambda_k)$, where $\bar{\alpha}_k \in [0,1)$ .
\end{algorithmic}
\end{algorithm}\medskip

\begin{theorem}\label{dgn-complexity}
Consider the iterate sequences of the Frank-Wolfe method with the $(\delta, L)$-oracle (Method \ref{fw-dgn}) $\{\lambda_k\}$ and $\{\tilde \lambda_k\}$ and the sequence of upper bounds $\{B_k\}$ on $h^*$, using the step-size sequence $\{\bar\alpha_k\}$.  For the auxiliary sequences $\{\alpha_k\}$ and $\{\beta_k\}$ given by (\ref{dos}), and for any $k \geq 0$, the following inequality holds:
\begin{equation}\label{fwdgn-ineq1}
B_k - h(\lambda_{k+1}) \leq \frac{B_k - h(\lambda_1)}{\beta_{k+1}} + \frac{\frac{1}{2}\mathrm{Diam}_Q^2\sum_{i = 1}^kL_i\frac{\alpha_i^2}{\beta_{i+1}}}{\beta_{k+1}} + \frac{\sum_{i = 1}^k\beta_{i+1}\delta_i}{\beta_{k+1}} \ .
\end{equation}\qed
\end{theorem}\medskip

As with Theorem \ref{fwapprox-complexity}, observe that the terms on the right-hand side of \eqref{fwdgn-ineq1} are composed of the exact terms which appear on the right-hand side of \eqref{fw-ineq1}, plus an additional term that is a function of $\delta_1, \ldots, \delta_k$. Unfortunately, Theorem \ref{dgn-complexity} implies an accumulation of errors for Method \ref{fw-dgn} under essentially any choice of step-size sequence $\{\bar{\alpha}_k\}$. Indeed, suppose that $\beta_i = O(i^\gamma)$ for some $\gamma \geq 0$, then $\sum_{i = 1}^k\beta_{i+1} = O(k^{\gamma + 1})$, and in the constant case where $\delta_i := \delta$, we have $\frac{\sum_{i = 1}^k\beta_{i+1}\delta_i}{\beta_{k+1}} = O(k\delta)$.  Therefore in order to achieve an $O\left(\frac{1}{k}\right)$ rate of convergence (for example with the step-size sequence \eqref{nemirovsteps}) we need $\delta = O\left(\frac{1}{k^2}\right)$. This negative result nevertheless contributes to the understanding of the merits and demerits of different first-order methods as follows.  Note that in \cite{devolderApprox} it is shown that the ``classical'' gradient methods (both primal and dual), which require solving a proximal projection problem at each iteration, achieve an $O\left(\frac{1}{k} + \delta\right)$ accuracy under the $(\delta, L)$-oracle model for  constant $(\delta, L)$.  On the other hand, it is also shown in \cite{devolderApprox} that all accelerated first-order methods (which also solve proximal projection problems at each iteration) generically achieve an $O\left(\frac{1}{k^2}+ k\delta\right)$ accuracy and thus suffer from an accumulation of errors under the $(\delta, L)$-oracle model.  As discussed in the Introduction herein, the Frank-Wolfe method offers two possible advantages over these proximal methods: (i) the possibility that solving the linear optimization subproblem is easier than the projection-type problem in an iteration of a proximal method, and/or (ii) the possibility of greater structure (sparsity, low rank) of the iterates.  In Figure \ref{fomTable} we summarize the cogent properties of these three methods (or classes of methods) under exact gradient computation as well as with the $(\delta, L)$-oracle model.  As can be seen from the table in Figure \ref{fomTable}, no single method dominates in the three categories of properties shown in the table; thus there are inherent tradeoffs among these methods/classes.\medskip

\begin{figure}
\centering
\begin{tabular}{|c|c||c|c|c|}
\hline
Method/ & Type of & Accuracy with & Accuracy with  &  Special Structure \\
Class & Subproblem & Exact Gradients & $(\delta, L)$-oracle   & of Iterates\\ \hline \hline
Frank-Wolfe & Linear Optimization & $O\left(1/k \right)$ & $O\left(1/k \right) + O(\delta k)$  & Yes \\ \hline
Classical Gradient & Prox Projection & $O\left(1/k \right)$ & $O(1/k ) + O(\delta )$  & No \\ \hline
Accelerated Gradient & Prox Projection & $O\left(1/k^2 \right)$ & $O\left(1/k^2 \right) + O(\delta k)$  & No \\ \hline
\end{tabular}
\caption{Properties of three (classes of) first-order methods after $k$ iterations.}\label{fomTable}
\end{figure}

\noindent {\bf Proof of Theorem \ref{dgn-complexity}:}
Note that \eqref{dgnCondition} and \eqref{dgnCondition2} with $\bar{\lambda} = \lambda$ imply that:
\begin{equation}\label{liberty}
h(\lambda) \leq h_{(\delta,L)}(\lambda) \leq h(\lambda) + \delta \ \ \text{ for all } \lambda \in Q \ .
\end{equation}
Recall properties \eqref{identity1} and \eqref{identity2} of the dual averages sequences \eqref{dos}. Following the proof of Theorem \ref{fw-complexity}, we then have for $i \ge 1$:
\begin{align*}
\beta_{i+1} h(\lambda_{i+1}) & \ge \beta_{i+1} \left[h_i+ g_i^T(\tilde\lambda_i-\lambda_i)\bar\alpha_i - \frac{1}{2}\bar\alpha_i^2L_i\mathrm{Diam}_Q^2 - \delta_i \right] \\
& = \beta_{i} h_i + (\beta_{i+1} - \beta_{i})h_i +  \beta_{i+1}\bar\alpha_ig_i^T(\tilde\lambda_i-\lambda_i) - \frac{1}{2}\beta_{i+1}\bar\alpha_i^2L_i\mathrm{Diam}_Q^2 - \beta_{i+1}\delta_i\\
& = \beta_{i} h_i + \alpha_{i}\left[h_i +  g_i^T(\tilde\lambda_i-\lambda_i)\right] - \frac{1}{2}\frac{\alpha_i^2}{\beta_{i+1}}L_i\mathrm{Diam}_Q^2  - \beta_{i+1}\delta_i\\
& \geq \beta_{i} h(\lambda_i) + \alpha_{i}B^w_i - \frac{1}{2}\frac{\alpha_i^2}{\beta_{i+1}}L_i\mathrm{Diam}_Q^2 - \beta_{i+1}\delta_i \ ,
\end{align*}
where the first inequality uses \eqref{dgnCondition2}, and the second inequality uses \eqref{liberty} and the definition of the Wolfe upper bound in Method \ref{fw-dgn}. The rest of the proof follows as in the proof of Theorem \ref{fw-complexity}. \qed

\section{Summary/Conclusions}\label{conclusions}
The Frank-Wolfe method is the subject of substantial renewed interest due to the relevance of applications (e.g., regularized regression, boosting/classification, matrix completion, image construction, other machine learning problems), the need in many applications for only moderately high accuracy solutions, the applicability of the method on truly large-scale problems, and the appeal of structural implications (sparsity, low-rank) induced by the method itself.  The method requires (at each iteration) the solution of a linear optimization subproblem over the feasible region of interest, in contrast to most other first-order methods which require (at each iteration) the solution of a certain projection subproblem over the feasible region defined by a strongly convex prox function.  As such,  the Frank-Wolfe method is particularly efficient in various important application settings including matrix completion.\medskip

In this paper we have developed new analysis and results for the Frank-Wolfe method.  Virtually all of our results are consequences and applications of Theorems \ref{fw-complexity} and \ref{gap-complexity}, which present computational guarantees for optimality gaps (Theorem \ref{fw-complexity}) and the ``FW gaps'' (Theorem \ref{gap-complexity}) for arbitrary step-size sequences $\{\bar\alpha_k\}$ of the Frank-Wolfe method.   These technical theorems are applied to yield computational guarantees for the well-studied step-size rule $\bar\alpha_k := \frac{2}{k+2}$ (Section \ref{vivek}), simple averaging (Section \ref{pointer}), and constant step-size rules (Section \ref{brick}).  The second set of contributions in the paper concern ``warm start'' step-size rules and computational guarantees that reflect the quality of the given initial iterate (Section \ref{sect-warmstart}) as well as the accumulated information about the optimality gap and the curvature constant over a sequence of iterations (Section \ref{voip}).  The third set of contributions concerns computational guarantees in the presence of an approximate solution of the linear optimization subproblem (Section \ref{approx-subprob}) and approximate computation of gradients (Section \ref{subsect-aspremont}).\medskip

We end with the following observation:  that the well-studied step-size rule $\bar\alpha_k := \frac{2}{k+2}$ does not require any estimation of the curvature constant $C_{h,Q}$ (which is generically not known).  Therefore this rule is in essence fully-automatically scaled as regards the curvature $C_{h,Q}$.  In contrast, the dynamic warm-start step-size rule \eqref{Aa}, which incorporates accumulated information over a sequence of iterates, requires updating estimates of the curvature constant $C_{h,Q}$ that satisfy certain conditions.  It is an open challenge to develop a dynamic warm-start step-size strategy that is automatically scaled and so does not require computing or updating estimates of $C_{h,Q}$.

\bibliographystyle{amsplain}
\bibliography{GF-papers-nips}

\appendix
\section{Appendix}\label{proofs}

\begin{proposition}\label{twobounds}
Let $B_k^w$ and $B_k^m$ be as defined in Section \ref{sectfw}. Suppose that there exists an open set $\hat{Q} \subseteq E$ containing $Q$ such that $\phi(x,\cdot)$ is differentiable on $\hat{Q}$ for each fixed $x \in P$, and that $h(\cdot)$ has the minmax structure \eqref{minmax_structure} on $\hat Q$ and is differentiable on $\hat Q$. Then it holds that:
\begin{equation*}
B_k^w \geq B_k^m \geq h^* \ .
\end{equation*} Furthermore, it holds that $B_k^w = B_k^m$ in the case when $\phi(x,\cdot)$ is linear in the variable $\lambda$.
\end{proposition}
\begin{proof}
It is simple to show that $B_k^m \geq h^*$.  At the current iterate $\lambda_k \in Q$, define $x_k \in \arg\min\limits_{x \in P}\phi(x, \lambda_k)$.  Then from the definition of $h(\lambda)$ and the concavity of $\phi(x_k, \cdot)$ we have:
\begin{equation}\label{sandman}
h(\lambda) \le \phi(x_k, \lambda) \leq \phi(x_k, \lambda_k) + \nabla_\lambda\phi(x_k, \lambda_k)^T(\lambda - \lambda_k) = h(\lambda_k) + \nabla_\lambda\phi(x_k, \lambda_k)^T(\lambda - \lambda_k) \ ,
\end{equation}whereby $\nabla_\lambda\phi(x_k, \lambda_k)$ is a subgradient of $h(\cdot)$ at $\lambda_k$.  It then follows from the differentiability of $h(\cdot)$ that $\nabla h(\lambda_k) = \nabla_\lambda\phi(x_k, \lambda_k)$, and this implies from \eqref{sandman} that:
\begin{equation}\label{sandman2}
\phi(x_k, \lambda) \leq h(\lambda_k) + \nabla h(\lambda_k)^T(\lambda - \lambda_k) \ .
\end{equation}
Therefore we have:
\begin{equation*}
B_k^m = f(x_k) = \max_{\lambda \in Q}\{\phi(x_k, \lambda)\} \leq \max_{\lambda \in Q}\{h(\lambda_k) + \nabla h(\lambda_k)^T(\lambda - \lambda_k)\} = B_k^w \ .
\end{equation*}If $\phi(x,\lambda)$ is linear in $\lambda$, then the second inequality in \eqref{sandman} is an equality, as is \eqref{sandman2}.
\end{proof}

\begin{proposition}\label{barL}
Let $C_{h, Q}$, $\mathrm{Diam}_Q$, and $L_{h,Q}$ be as defined in Section \ref{sectfw}.  Then it holds that  $C_{h, Q} \le L_{h,Q}(\mathrm{Diam}_Q)^2 $.
\end{proposition}
\begin{proof}
Since $Q$ is convex, we have $\lambda + \alpha (\tilde \lambda - \lambda) \in Q$ for all $\lambda, \tilde{\lambda} \in Q$ and for all $\alpha \in [0,1]$. Since the gradient of $h(\cdot)$ is Lipschitz, from the fundamental theorem of calculus we have:
\begin{align*}
h(\lambda + \alpha (\tilde \lambda - \lambda)) & = h(\lambda) + \nabla h(\lambda)^T(\alpha (\tilde \lambda - \lambda)) +\int\limits_0^1[\nabla h(\lambda + t \alpha (\tilde \lambda - \lambda)) - \nabla h (\lambda)]^T(\alpha(\tilde \lambda - \lambda)) dt \\
& \ge h(\lambda) + \nabla h(\lambda)^T(\alpha (\tilde \lambda - \lambda)) -\int\limits_0^1 \|\nabla h(\lambda + t \alpha (\tilde \lambda - \lambda)) - \nabla h (\lambda)\|_*(\alpha)\|\tilde \lambda - \lambda\| dt \\
& \ge h(\lambda) + \nabla h(\lambda)^T(\alpha (\tilde \lambda - \lambda)) -\int\limits_0^1 L_{h, Q} \|(t \alpha)(\tilde \lambda - \lambda)\|( \alpha)\|\tilde \lambda - \lambda\| dt \\
& = h(\lambda) + \nabla h(\lambda)^T(\alpha (\tilde \lambda - \lambda)) - \frac{\alpha^2}{2}L_{h, Q}\|(\tilde \lambda - \lambda)\|^2 \\
&\geq h(\lambda) + \nabla h(\lambda)^T(\alpha (\tilde \lambda - \lambda)) - \frac{\alpha^2}{2}L_{h, Q}(\mathrm{Diam}_Q)^2 \ ,
\end{align*} whereby it follows that $C_{h, Q} \le L_{h,Q}(\mathrm{Diam}_Q)^2 $.
\end{proof}

\begin{proposition}\label{boringbound1} For $k\ge 0$ the following inequality holds:
\begin{equation*}
\sum_{i=0}^k \frac{i+1}{i+2} \le \frac{(k+1)(k+2)}{k+4} \ .
\end{equation*}
\end{proposition}
\begin{proof}The inequality above holds at equality for $k=0$.  By induction, suppose the inequality is true for some given $k\ge 0$, then
\begin{equation}\begin{array}{rl}\label{boring1a}
\sum_{i=0}^{k+1} \frac{i+1}{i+2} & = \sum_{i=0}^{k} \frac{i+1}{i+2} + \frac{k+2}{k+3} \\ \\
& \le \frac{(k+1)(k+2)}{k+4} + \frac{k+2}{k+3} \\ \\
& = (k+2)\left[\frac{k^2+5k+7}{k^2 + 7k +12}\right] \ .
\end{array}\end{equation}Now notice that $$(k^2 +5k+7)(k+5) = k^3 + 10k^2 +32k + 35 < k^3 + 10k^2 +33k + 36 =  (k^2 + 7k +12)(k+3) \ , $$ which combined with \eqref{boring1a} completes the induction.
\end{proof}

\begin{proposition}\label{boringbound2}For $k\ge 1$ let $\bar\alpha:= 1-\frac{1}{\sqrt[k]{k+1}}$.  Then the following inequalities holds:
\begin{itemize}
\item[(i)] $\displaystyle\frac{\ln(k+1)}{k} \ge \bar\alpha $ , and
\item[(ii)] $(k+1)\bar\alpha \ge 1 $ .
\end{itemize}
\end{proposition}
\begin{proof}  To prove {\em (i)}, define $f(t):= 1-e^{-t}$, and noting that $f(\cdot)$ is a concave function, the gradient inequality for $f(\cdot)$ at $t=0$ is
$$ t \ge 1-e^{-t} \ . $$
Substituting $t=\frac{\ln(k+1)}{k} $ yields $$ \frac{\ln(k+1)}{k} = t \ge 1-e^{-t} = 1 - e^{-\frac{\ln(k+1)}{k}} = 1- \frac{1}{\sqrt[k]{k+1}} = \bar\alpha \ . $$
Note that {\em (ii)} holds for $k = 1$, so assume now that $k \geq 2$. To prove {\em (ii)} for $k \geq 2$, substitute $t = -\frac{\ln(k+1)}{k}$ into the gradient inequality above to obtain $-\frac{\ln(k+1)}{k} \geq 1 - (k+1)^{\frac{1}{k}}$ which can be rearranged to:
\begin{equation}\label{yeahyeah}
(k+1)^{\frac{1}{k}} \geq 1 + \frac{\ln(k+1)}{k} \geq 1 + \frac{\ln(e)}{k} = 1 + \frac{1}{k} = \frac{k+1}{k} \ .
\end{equation}
Inverting \eqref{yeahyeah} yields:
\begin{equation}\label{yeahyeahyeah}
(k+1)^{-\frac{1}{k}} \leq \frac{k}{k+1} = 1 - \frac{1}{k+1} \ .
\end{equation}
Finally, rearranging \eqref{yeahyeahyeah} and multiplying by $k+1$ yields {\em (ii)}.
\end{proof}

\begin{proposition}\label{boringbound3}
For any integers $\ell, k$ with $2 \leq \ell \leq k$, the following inequalities hold:
\begin{equation}\label{boringintegral1}
\ln\left(\frac{k+1}{\ell}\right) \leq \sum_{i = \ell}^k\frac{1}{i} \leq \ln\left(\frac{k}{\ell - 1}\right) \ ,
\end{equation}
and
\begin{equation}\label{boringintegral2}
\frac{k - \ell + 1}{(k+1)\ell} \leq \sum_{i = \ell}^k\frac{1}{i^2} \leq \frac{k - \ell + 1}{k(\ell-1)} \ ,
\end{equation}
\end{proposition}
\begin{proof}
\eqref{boringintegral1} and \eqref{boringintegral2} are specific instances of the following more general fact: if $f(\cdot): [1, \infty) \to \mathbb{R}_+$ is a monotonically decreasing continuous function, then
\begin{equation}\label{boringintegral3}
\int_{\ell}^{k+1}f(t)dt \leq \sum_{i = \ell}^kf(i) \leq \int_{\ell - 1}^kf(t)dt \ .
\end{equation}
It is easy to verify that the integral expressions in \eqref{boringintegral3} match the bounds in \eqref{boringintegral1} and \eqref{boringintegral2} for the specific choices of $f(t) = \frac{1}{t}$ and $f(t) = \frac{1}{t^2}$, respectively.
\end{proof}

\end{document}